\documentclass[10pt,twocolumn,letterpaper]{article}
\usepackage[
  top=0.75in,
  bottom=0.6in,
  left=0.7in,
  right=0.7in,
  columnsep=0.15in
]{geometry}
\usepackage{cite}
\usepackage{amsmath,amssymb,amsfonts,amsthm}
\usepackage{mathtools, mathrsfs, xcolor}
\usepackage{graphicx}
\usepackage{algorithm,algorithmic}
\usepackage{hyperref}
\usepackage{textcomp}
\def\BibTeX{{\rm B\kern-.05em{\sc i\kern-.025em b}\kern-.08em
    T\kern-.1667em\lower.7ex\hbox{E}\kern-.125emX}}
\newtheorem{theorem}{Theorem}
\newtheorem{proposition}[theorem]{Proposition}
\newtheorem{lemma}[theorem]{Lemma}

\newtheorem{remark}{Remark}
\newtheorem{claim}{Claim}
\newtheorem{definition}{Definition}
\begin{document}
\title{Intrinsic Contraction Theory and Stabilization of Probability Measures on the Wasserstein Space}
\author{Min Kim$^*$, Joshua D. Ibrahim$^*$, and Soon-Jo Chung}

\date{}
\maketitle
{
\renewcommand{\thefootnote}{}
\footnotetext[0]{$^*$These authors contributed equally.\par
This work has been submitted to the IEEE for possible publication. Copyright may be transferred without notice, after which this version may no longer be accessible.\par
This work was supported by DARPA Learning Introspective Control (LINC) and Technology Innovation Institute (TII).\par
All authors are with the California Institute of Technology, Pasadena, CA 91125 USA (e-mail: \{mink, jdibrahi, sjchung\}@caltech.edu).}
}
\pagestyle{plain}
\thispagestyle{plain} 
\begin{abstract}
Deterministic flows can contract Wasserstein distances between probability laws without being contractive in point-state space. We develop a differential contraction theory for pushforward flows on forward-invariant subsets of the 2-Wasserstein space that quotients out the gauge freedom of the continuity equation. Weighted Helmholtz--Hodge decomposition and the Wasserstein contingent cone define the weighted horizontal tangent cone (whtc). Requiring the contraction quadratic form to be nonpositive only on the whtc yields the Intrinsic Wasserstein Contraction Condition (IWCC), which generalizes the classical contraction LMI and recovers it on Dirac measures. An almost-everywhere Euler approximation based on the superposition principle and Banach-valued Lebesgue differentiation theorem proves that IWCC is sufficient for intrinsic Wasserstein contraction. Law contraction without point-state contraction is illustrated by a rotationally symmetric shear example. Finally, a new comparison lemma yields an incremental Wasserstein ISS theorem; for Brockett's integrator, we prove exponential stabilization of a class of axisymmetric laws to the uniform law on a circle via a smooth feedback.
\end{abstract}


\section{Introduction}
Contraction theory~\cite{LoSl1998, TsCh2021, Bu2026} of Lohmiller and Slotine, which is an infinitesimal analogue of Lyapunov stability theory, studies a given dynamical system $\dot{\mathbf{x}} = f(t,\mathbf{x})$ on a Riemannian manifold $(\mathscr{M},g)$ by finding an auxiliary metric $\tilde g$ such that the infinitesimal energy $\frac{1}{2}\tilde g(d\mathbf{x},d\mathbf{x})$ decays exponentially along the flow of the dynamical system. It has become a powerful tool in control theory and dynamical systems theory for establishing nonlinear stability. The central idea is to compare two solution trajectories by choosing an arbitrary smooth path connecting their initial conditions and then transporting this path through the dynamics. Exponential decay of the infinitesimal energy implies exponential decay of the $\tilde g$-length of the transported path, and hence of the $\tilde g$-distance between the two solutions. Assuming the auxiliary metric $\tilde g$ is uniformly bounded below by $g$, this $\tilde g$-distance estimate yields a corresponding bound on the original Riemannian distance between the solutions.

\begin{figure*}[!t]
    \centering
    \includegraphics[width=\linewidth]{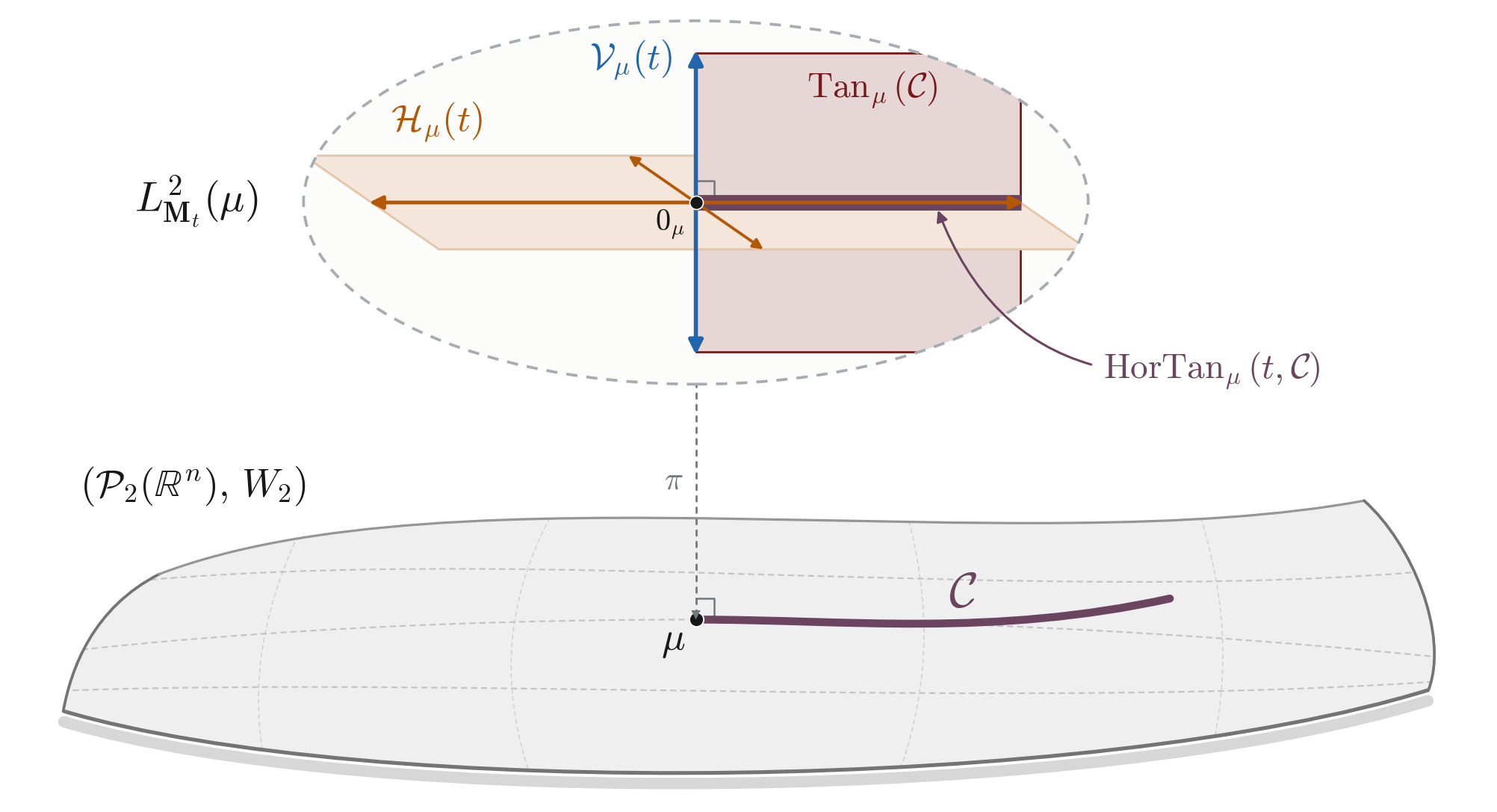}
    \caption{Schematic geometry of the weighted horizontal tangent cone (whtc) at $\mu \in \mathcal{C}$ and at time $t \ge 0$. The lower surface represents the Wasserstein space $(\mathcal{P}_2(\mathbb{R}^n), W_2)$, with $\mathcal{C}$ drawn as a one-dimensional curve with endpoint $\mu$. The upper oval represents the weighted velocity space $L^2_{\mathbf{M}_t} (\mu)$, viewed formally as a fiber over $\mu$, and displays its weighted Helmholtz--Hodge decomposition $L^2_{\mathbf{M}_t}(\mu)=\mathcal{H}_\mu (t) \oplus \mathcal{V}_\mu (t)$.
    The Wasserstein contingent cone $\operatorname{Tan}_\mu (\mathcal{C})$ consists of directions tangent to $\mathcal{C}$, and its horizontal part is the whtc, which is defined by $\operatorname{HorTan}_\mu(t,\mathcal{C})\coloneqq  \mathcal{H}_\mu(t) \cap \operatorname{Tan}_\mu(\mathcal{C})$. }
    \label{fig:whtc}
\end{figure*}
This paper generalizes contraction theory to flows on the 2-Wasserstein space $\mathcal{P}_2(\mathbb{R}^n)$. Our contributions are threefold. First, we develop a natural extension of contraction theory to pushforward-induced flows on the 2-Wasserstein space $\mathcal{P}_2(\mathbb{R}^n)$ by directly using the Benamou--Brenier formula~\cite{BeBr2000} and the formal Riemannian viewpoint on $\mathcal{P}_2(\mathbb{R}^n)$ provided by Otto calculus~\cite{Ot2001}. This results in a novel contraction condition, called the \emph{Intrinsic Wasserstein Contraction Condition} (IWCC), which generalizes the classical contraction condition of Lohmiller and Slotine; the latter is recovered when $\mathcal C=\{\delta_{\mathbf{x}}:\mathbf{x}\in\mathbb R^n\}$, the subset of Dirac measures, which is identified with the original state space $\mathbb{R}^n$. A key analytical step in the proof of the IWCC combines the superposition principle for the continuity equation with Banach-valued Lebesgue differentiation theorem, yielding an Euler-type first-order approximation of curves in $\mathcal{C}$ and showing that their velocity fields belong to the relevant Wasserstein contingent cones almost everywhere. Second, we further extend the theory to controlled systems $\dot{\mathbf{x}}=f(t,\mathbf{x},\mathbf{u})$. Given a control homotopy $\bar{\mathbf u}$ between two given control fields, we introduce the control IWCC (cIWCC) for $(\mathbf M,\bar{\mathbf u})$. Under the cIWCC, we prove an incremental Wasserstein ISS estimate. The proof of the ISS estimate relies on a square-root Dini comparison lemma involving a function that is lower semicontinuous from the left and a Lebesgue-measurable forcing term. 
Finally, we apply the ISS theorem to Brockett's integrator and construct a smooth feedback that exponentially stabilizes an axisymmetric invariant probability law supported on a circle.

We develop a contraction criterion for the \emph{intrinsic} Wasserstein distance on a \emph{forward-invariant subset} $\mathcal{C}$ of $\mathcal{P}_2(\mathbb{R}^n)$. A careful analysis reveals the \emph{weighted horizontal tangent cone} (whtc) as a natural test space for the IWCC. The elements of the whtc are obtained in two steps. First, the weighted Helmholtz--Hodge decomposition quotients out the infinitesimal stabilizer directions of a probability measure $\mu \in \mathcal{C}$ by considering minimizers of a kinetic energy functional whose mass is weighted by $\mathbf{M}(t,\mathbf{x})$. Moreover, because we are considering the intrinsic Wasserstein distance on $\mathcal{C}$, the relevant velocity fields may be further assumed to lie in the Wasserstein contingent cone of $\mathcal C$. These two conditions jointly define the whtc at $\mu$. Figure~\ref{fig:whtc} illustrates this geometry schematically. The weighted Helmholtz--Hodge decomposition decomposes the fiber over $\mu$ as $L^2_{\mathbf{M}_t}(\mu)=\mathcal{H}_\mu (t) \oplus \mathcal{V}_\mu (t)$, while restricting the analysis to $\mathcal{C}$ naturally leads to the Wasserstein contingent cone $\operatorname{Tan}_\mu (\mathcal{C})$. Accordingly, the relevant test directions for intrinsic Wasserstein contraction form the whtc, defined by $\operatorname{HorTan}_\mu(t,\mathcal{C})\coloneqq  \mathcal{H}_\mu(t) \cap \operatorname{Tan}_\mu(\mathcal{C})$. We further justify this choice by showing, by means of an example, that neither $\mathcal{H}_\mu(t)$ nor $\operatorname{Tan}_\mu (\mathcal{C})$ is adequate for the desired contraction, thereby leading to the IWCC theorem with the whtc as its test space.

\paragraph{Related Works}
Contraction theory has been extended beyond requiring full contraction of the original state dynamics. Partial contraction uses a contracting auxiliary system to establish convergence toward a prescribed property~\cite{WaSl2005}, and horizontal contraction requires contraction only in selected horizontal directions~\cite[Theorem~4]{FoSe2014}. While \cite{WaSl2005} proves that point-state trajectories converge to a prescribed property, our theory proves an exponential upper bound on the 2-Wasserstein distance between two measure-valued trajectories and does not introduce a contracting auxiliary system or prescribe a target property. Our construction differs from \cite{FoSe2014} in the origin and meaning of the horizontal space. In \cite{FoSe2014}, a horizontal--vertical decomposition is specified on the tangent bundle of a smooth state manifold, and the selected vertical state-space directions are ignored by the horizontal Finsler--Lyapunov function. Moreover, \cite{FoSe2014} does not treat probability distributions as states. In our setting, the vertical space at each probability distribution $\mu$ is instead induced by the continuity equation, and the auxiliary metric selects the minimum kinetic energy representative of each infinitesimal law movement. Moreover, since $\mathcal C$ need not be a smooth subset, the infinitesimal law movements are described by the Wasserstein contingent cone instead of a tangent space, which is further restricted by our horizontal space.

On the Wasserstein side, Bouvrie and Slotine~\cite{BoSl2019} studied Wasserstein contraction for stochastic systems, showing that contraction of the underlying noise-free dynamics can yield bounds on Wasserstein distances between the laws of solutions to SDEs. Here, the Wasserstein contraction was derived as a consequence of the classical contraction condition checked on virtual displacements in the whole Euclidean tangent space, whereas our work directly tests the contraction quadratic form on the virtual displacements in the Wasserstein space, and further generalizes the condition by defining the whtc. We also note that, in contrast to gradient-flow-based Wasserstein contraction results~\cite[Chapter~11]{AmGi2008},~\cite{CaMc2006}, our theory does not require the dynamics to arise from an underlying energy functional.

On the application side, contraction-regularized learning was used to learn fault-indexed Perron--Frobenius operators, which were then used for fault inference and recovery~\cite{Ib2026p}; Wasserstein contraction was also used to provide a necessary condition for fault detectability~\cite{Ib2026d}. We also note that mean-square stochastic contraction was applied to observer design~\cite{DaCh2015}, resulting in a stochastic stability theorem for the observer.

\paragraph{Basic Definitions}
We now begin by recalling the basic definitions~\cite{FiGl2023, Vi2009}. Let $(X,d)$ be a nonempty separable metric space, let $\mathcal{B}(X)$ denote its Borel $\sigma$-algebra, and let $\mathcal{P}(X)$ be the set of probability measures on $(X,\mathcal{B}(X))$. The space $\mathcal{P}_2(X)$ consists of Borel probability measures with finite second moment, i.e., for any fixed $\mathbf{x}_0 \in X$,
\begin{equation}
    \mathcal{P}_2(X) \coloneqq \left\{ \mu \in \mathcal{P}(X): \int_X d(\mathbf{x},\mathbf{x}_0)^2 d\mu(\mathbf{x}) < \infty \right\}.
\end{equation}
The above condition holds for every $\mathbf{x}_0 \in X$ whenever it holds for one such point. For $\mu,\nu \in \mathcal{P}(X)$, let $\Pi(\mu,\nu)$ denote the set of Borel probability measures on $X \times X$ whose first and second marginals are $\mu$ and $\nu$, respectively. The 2-Wasserstein distance between $\mu$ and $\nu$ is defined by
\begin{equation}
    W_2(\mu,\nu) \coloneqq \left( \inf_{\pi \in \Pi(\mu,\nu)} \int_{X \times X} d(\mathbf{x},\mathbf{y})^2 d\pi(\mathbf{x},\mathbf{y}) \right)^{1/2}.
\end{equation}
This defines a metric on $\mathcal{P}_2(X)$~\cite{ClDe2008}. Moreover, if $(X,d)$ is complete, then $(\mathcal{P}_2(X), W_2)$ is complete as well. In the remainder of the paper, we take $X=\mathbb{R}^n$. We therefore often suppress the dependence on $X$ and simply write $\mathcal{B}$, $\mathcal{P}$, $\mathcal{P}_2$, and so on. Recall that a curve $(\mu_s)_{s \in [0,1]}$ in $\mathcal{P}_2$ is called a (constant speed) geodesic if $W_2(\mu_s,\mu_t)=|t-s|W_2(\mu_0,\mu_1)$ for all $s, t$. A subset $\mathcal{C}\subseteq\mathcal{P}_2$ is called weakly displacement convex if for any $\mu,\nu\in\mathcal{C}$ there exists a geodesic from $\mu$ to $\nu$ which stays in $\mathcal{C}$.

\subsection{Motivation: An Example}
\label{subsec:motivating_example}
We provide a simple example to motivate the constructions in this paper. The example is a rotationally symmetric shear on $\mathbb{R}^2$, where the radial variable contracts, while the angular velocity depends on the radius. This example will be conceptually generalized further in Claim~\ref{claim:example_compactgroup}.

Fix $\alpha>0$ and $\beta\neq 0$. Let $\chi \ge 0$ be a smooth function on $\mathbb R$ such that $\chi(r)=0$ for $r\le \frac 1 8$, $\chi(r)=1$ for $r\ge \frac 1 4$. We consider the following dynamics in polar coordinates:
\begin{equation}
\label{eq:motivating_polar_system}
\begin{aligned}
    \dot r&=-\alpha\chi(r)(r-1),\\
    \dot\theta&=1+\beta\chi(r)(r-1).
\end{aligned}
\end{equation}
Since $\chi$ vanishes near the origin, the associated Cartesian vector field is $f(x_1, x_2)=(-x_2, x_1)$ near $x=0$, and is therefore smooth on $\mathbb{R}^2$.

Let $A = \left\{\mathbf{x} \in \mathbb{R}^2: \frac 1 2 < |\mathbf{x}| < 2 \right \}$, and let $\mathcal C_A$ be the set of rotationally-invariant probability measures concentrated on $A$. The set $A$ is forward invariant, and so is $\mathcal C_A$. On $A$, the radial dynamics reduce to $\dot r=-\alpha(r-1)$, $\alpha>0$. Consequently, we would expect exponential incremental stability for measures $\mu,\nu \in \mathcal C_A$. However, this does not entail pointwise contraction of the planar system. Indeed, if $\mathbf{x}_1(0)=e^{i\theta_1}$ and $\mathbf{x}_2(0)=e^{i\theta_2}$ lie on the unit circle, then $r_1(t)=r_2(t)=1$ and $|\mathbf{x}_1(t)-\mathbf{x}_2(t)|=|\mathbf{x}_1(0)-\mathbf{x}_2(0)|$ for all $t\ge 0$. Moreover, even if two points in $A$ start on the same radial ray but have different radii, we have $\theta_1(t)-\theta_2(t) = \frac{\beta}{\alpha}(r_1(0)-r_2(0))(1-e^{-\alpha t}) \to \frac \beta \alpha (r_1(0)-r_2(0))$ as $t \to \infty$, and this limit need not be an integer multiple of $2\pi$.
Note that the angular component of the vector field is invisible to the laws $\mu \in \mathcal{C}_A$. This is the phenomenon encoded by the horizontal--vertical decomposition below in Lemma~\ref{lemma:vertical_horizontal_decomp}. Figure~\ref{fig:rot_sym_shear} illustrates this contrast between the pointwise dynamics in \eqref{eq:motivating_polar_system} and the induced dynamics on $\mathcal{C}_A$.

The angular shear is irrelevant in the Wasserstein picture even though it cannot be ignored in pointwise contraction analysis. This motivates a contraction theory for the Wasserstein space where the contraction property is tested only along the ``relevant'' directions. Later, we will define the \emph{weighted horizontal tangent cone} $\operatorname{HorTan}_\mu(t,\mathcal{C})$ in Definition~\ref{def:weighted_horizontal_tangent_cone} to formalize this concept of relevant directions.

\begin{figure}[!t]
    \centering
    \includegraphics[width=\linewidth]{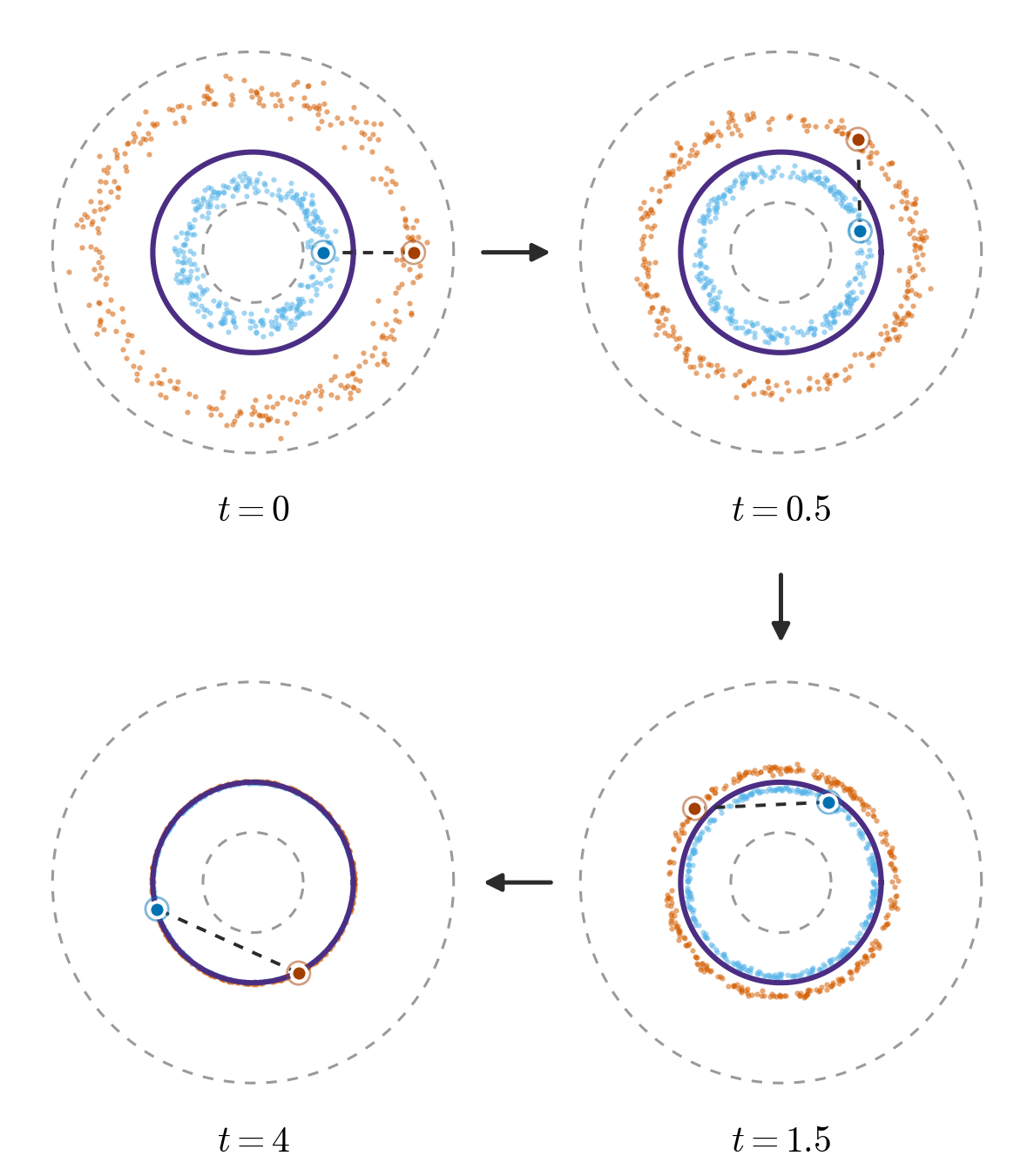}
    \vspace{-5pt}
    \caption{Snapshots of the rotationally symmetric shear dynamics of Section~\ref{subsec:motivating_example} on the annulus $A = \left\{\mathbf{x} \in \mathbb{R}^2: \frac 1 2 < |\mathbf{x}| < 2 \right \}$, with $\alpha=1$ and $\beta = 2$, shown at times $t = 0$, $0.5$, $1.5$, and $4$. Point clouds sampled from two rotationally invariant laws $\mu$ (orange, supported on $\{r > 1\}$) and $\nu$ (blue, supported on $\{r < 1\}$) are pushed forward by the flow. The pushforwards $\mathsf{S}_{0,t}\mu$ and $\mathsf{S}_{0,t}\nu$ exponentially converge to the uniform law on the circle $S^1$ (purple). In contrast, the two highlighted point trajectories initially placed on the same ray $\theta = 0$ do not converge to each other due to the angular shear. The angular shear acts only along infinitesimal stabilizer directions of the laws in $\mathcal{C}_{A}$.}  
\label{fig:rot_sym_shear}
\end{figure}

\section{Contraction Theory on the 2-Wasserstein Space}
\label{sec:wasserstein_contraction}
In this section, we rigorously formulate the general principles of intrinsic contraction theory on subsets $\mathcal{C} \subseteq \mathcal{P}_2$ of the 2-Wasserstein space.

\textbf{In Sections~\ref{sec:wasserstein_contraction} and \ref{sec:iwcc_pointwise}, we work under the following regularity assumptions.}
Let $f:[0,\infty)\times\mathbb{R}^n\to\mathbb{R}^n$ be continuous, and suppose that its spatial Jacobian $\frac{\partial f}{\partial \mathbf{x}}(t,\mathbf{x})$ is continuous. We assume that $f$ is forward complete, and we denote its forward flow by $\Phi_{r,t}$, where $0\le r\le t<\infty$. Let $\mathbf{M}:[0,\infty)\times\mathbb{R}^n\to\mathcal{S}_{++}^n$ be a continuously differentiable map into the set of positive definite real symmetric matrices.

\subsection{Problem Setup}
We denote the pushforward operators on $\mathcal{P}$ as $\mathsf{S}_{r,t}\coloneqq (\Phi_{r,t})_{\#}$. When we need to specify the underlying vector field $f(t,\mathbf{x})$, we write $\mathsf{S}^f_{r,t}$.
We lift $\mathbf{M}$ to a subset $\mathcal{C}$ of the Wasserstein space $\mathcal{P}_2$ in order to define a weighted analogue of the 2-Wasserstein distance (Definition~\ref{def:weighted_wasserstein_distance}); the definition is motivated by the Benamou--Brenier formula.
\begin{definition}[Admissible pair]
A family $(\mu_s,v_s)_{s\in[0,1]}$ is called an admissible pair if $s\mapsto \mu_s$ is a narrowly continuous map into $\mathcal{P}_2$, the vector field $v:[0,1]\times \mathbb{R}^n\to\mathbb{R}^n$ is Borel measurable, and $\int_0^1 \|v_s\|_{L^2(\mu_s)} ds < \infty$. The pair is also required to satisfy the continuity equation
\begin{equation*}
    \frac{\partial \mu_s}{\partial s}+\operatorname{div}\left(v_s\mu_s\right)=0
\end{equation*}
on $(0,1)\times\mathbb{R}^n$ in the weak sense: for every $\varphi \in C_c^\infty\left((0,1)\times\mathbb{R}^n\right)$,
\begin{equation*}
    \int_0^1 \int_{\mathbb{R}^n} \left( \partial_s \varphi(s,\mathbf{x}) +  \nabla \varphi(s,\mathbf{x}) \cdot v_s(\mathbf{x}) \right) d\mu_s(\mathbf{x}) ds = 0.
\end{equation*}
\end{definition}
If the pair $(\mu_\cdot,v_\cdot)$ is admissible, then $W_2(\mu_s, \mu_t) \le \int_s^t \|v_r\|_{L^2(\mu_r)}dr$ for all $0 \le s \le t \le 1$. In particular, $s \mapsto \mu_s\in\mathcal{P}_2$ is absolutely continuous.

The weighted distance below appears in Lisini~\cite[Corollary~2.5]{Li2009} without the restriction $\mu_\cdot \subseteq \mathcal{C}$. In this paper, we focus on the intrinsic distance on $\mathcal{C}$ and hence we impose $\mu_\cdot \subseteq \mathcal{C}$.
\begin{definition}
\label{def:weighted_wasserstein_distance}
Fix $t\ge0$, and let $(\mu_s)_{s\in[0,1]}$ be a narrowly continuous curve in $\mathcal{P}_2$. We define the $\mathbf{M}_t$-weighted action of $\mu_\cdot$ by
\begin{equation}
\label{eq:weighted_action}
\mathcal{A}_{\mathbf{M}_t}(\mu_\cdot)
    \coloneqq \inf_v \int_0^1 \!\! \int_{\mathbb{R}^n} v_s(\mathbf{x})^\top \mathbf{M}_t(\mathbf{x}) v_s(\mathbf{x}) d\mu_s(\mathbf{x}) ds,
\end{equation}
where $\mathbf{M}_t(\mathbf{x}) \coloneqq \mathbf{M}(t, \mathbf{x})$ and the infimum is taken over all vector fields $v$ such that $(\mu_s,v_s)_{s\in[0,1]}$ is an admissible pair. We allow the possibility of $\mathcal{A}_{\mathbf{M}_t} = \infty$.

For $\mu_0,\mu_1\in\mathcal{C}$ where $\mathcal{C} \subseteq \mathcal{P}_2$, we define
\begin{equation}
    W_{2, \mathcal{C},\mathbf{M}_t}^2(\mu_0,\mu_1)
    \coloneqq \inf_{\mu_\cdot \subseteq \mathcal{C}} \mathcal{A}_{\mathbf{M}_t}(\mu_\cdot),
\end{equation}
where the infimum is taken over all narrowly continuous curves $(\mu_s)_{s\in[0,1]}$ in $\mathcal{C}$ from $\mu_0$ to $\mu_1$. We allow the possibility of $W^2_{2, \mathcal{C}, \mathbf{M}_t} = \infty$. If $\mathcal{C} = \mathcal{P}_2$ or $\mathbf{M}_t \equiv I_n$, we omit it from the notation. We call $W_{2,\mathcal{C}}$ as the intrinsic (2-)Wasserstein distance on $\mathcal{C}$.
\end{definition}
\begin{remark}
Note that if $\mathcal{C} = \mathcal{P}_2$ and $\mathbf{M}_t \equiv I_n$, we have $W_{2,\mathcal{C},\mathbf{M}_t} \equiv W_2$ by the Benamou--Brenier formula~\cite[Theorems~7.2.2 and 8.3.1]{AmGi2008}. Also we have $W_{2,\mathcal{C}} \ge W_2$ on $\mathcal{C}$ by Benamou--Brenier. Note that the role of $W_{2,\mathcal C}$ relative to $W_2$ is analogous to the role of intrinsic Riemannian distance on an embedded submanifold $\mathscr M\subseteq \mathbb R^d$ relative to the ambient Euclidean distance. If $\mathcal{C}$ is weakly displacement convex, then we have $W_{2,\mathcal{C}} = W_2$ on $\mathcal{C}$.
\end{remark}

In the rest of this subsection, we prove that $W_{2,\mathcal{C},\mathbf{M}_t}$ is indeed a distance; more specifically, it is an extended metric on $\mathcal{C}$.
\begin{lemma}
\label{lemma:boundary_terms}
Let $(\mu_s, v_s)_{s \in [0,1]}$ be an admissible pair. Suppose $\Psi \in C^\infty([0,1]\times \mathbb{R}^n)$ and suppose that we have a constant $R>0$ such that $\Psi(s,\mathbf{x})=0$ whenever $|\mathbf{x}|>R$.
Then, \begin{equation}
    \int_0^1 \!\!\int_{\mathbb{R}^n} (\partial_s \Psi +\nabla_{\mathbf{x}}  \Psi \cdot v_s) d\mu_s ds = \int_{\mathbb{R}^n} \Psi_1 d\mu_1 - \int_{\mathbb{R}^n}\Psi_0 d\mu_0 .
\end{equation}
\end{lemma}
\begin{proof}
Because $\Psi$ is uniformly continuous on the compact set $[0,1] \times \{\mathbf{x}:|\mathbf{x}|\le R\}$ and $\mu_\cdot$ is narrowly continuous, $s \mapsto \int _{\mathbb{R}^n} \Psi(s,\cdot)d \mu_s$ is continuous. Plugging in $\varphi(s,\mathbf{x}) = \eta(s) \Psi (s,\mathbf{x})$ to the continuity equation for $\eta \in C_c^\infty ((0,1))$, we have \begin{equation*}
\begin{split}
    \int_0^1 \frac{\partial \eta}{\partial s}&\int_{\mathbb{R}^n}\!\Psi(s,\mathbf{x}) d \mu_s ds \\&+\int_0^1 \eta(s) \int_{\mathbb{R}^n} \! (\partial_s \Psi + \nabla_{\mathbf{x}} \Psi\cdot v_s )d\mu_s ds =0.
\end{split}
\end{equation*}
Plug in $\eta = \rho_\epsilon * \mathbf{1}_{[a,b]}$, $0<a<b<1$, with $\rho$ being the standard mollifier on $\mathbb{R}$~\cite[Appendix~C.5]{Ev2010}. Taking the limit $\epsilon \to 0+$ by using the continuity of $s \mapsto \int _{\mathbb{R}^n} \Psi(s,\cdot)d \mu_s$ for the first term and the dominated convergence theorem for the second term gives $\int_{\mathbb{R}^n}\Psi_a d \mu_a - \int_{\mathbb{R}^n} \Psi_b d \mu_b + \int_0^1 \mathbf{1}_{[a,b]} \int_{\mathbb{R}^n}(\partial_s \Psi + \nabla_{\mathbf{x}} \Psi\cdot v_s )d\mu_s \, ds =0 $. Taking the limits $a\to0+$ and then $b \to 1-$ concludes the proof.
\end{proof}

\begin{proposition} \label{prop:extended_metric}
$W_{2,\mathcal{C}, \mathbf{M}_t}$ is an extended metric on the set $\mathcal{C}$, i.e., $(\mathcal{C}, W_{2,\mathcal{C}, \mathbf{M}_t})$ satisfies the axioms of metric space except $W_{2,\mathcal{C}, \mathbf{M}_t}$ may take the value $\infty$.
\end{proposition}
\begin{proof}
Let $\mu,\nu,\lambda \in \mathcal{C}$. $W_{2,\mathcal{C},\mathbf{M}_t}(\mu,\nu) \ge 0$ is clear. We have $W_{2,\mathcal{C},\mathbf{M}_t}(\mu,\mu)=0$ because the pair $(s\mapsto \mu,s \mapsto 0)$ is admissible: we have $\int_0^1 \!\!\int_{\mathbb{R}^n} \!\partial_s \varphi d\mu ds = 0$ for any $\varphi \in C_c^\infty ((0,1)\times \mathbb{R}^n)$ by Fubini. If $(\mu_s, v_s)_{s \in [0,1]}$ is admissible, then $(\mu_{1-s},-v_{1-s})_{s\in[0,1]}$ is also admissible. Therefore, the $\mathbf{M}_t$-weighted action is invariant under the time flip $s \mapsto 1-s$; hence symmetry follows.

We now prove positivity. Suppose $W_{2,\mathcal{C},\mathbf{M}_t}(\mu,\nu)=0$, and let $\phi \in C_c^\infty(\mathbb{R}^n)$. Choose any admissible pair $(\mu_\cdot,v_\cdot)$ with $\mu_\cdot \subseteq \mathcal{C}$, $\mu_0=\mu$, and $\mu_1= \nu$. By Lemma~\ref{lemma:boundary_terms}, we have $\int \phi d \nu - \int \phi d \mu = \int_0^1\!\!\int\! \nabla \phi \cdot v_s d\mu_s ds$. Hence, by Cauchy--Schwarz,
$\left| \int \phi d \nu - \int \phi d\mu\right| \le \left\| \mathbf{M}_t^{-1/2} \nabla \phi\right\|_{\sup}\left(\int_0^1\!\!\int_{\mathbb{R}^n} \left\|\mathbf{M}_t^{1/2}v_s\right\|^2 d\mu_s ds\right)^{1/2}$.
Taking infimums, we conclude that $\int \phi d \nu = \int \phi d\mu$. Since $\phi \in C_c^\infty (\mathbb{R}^n)$ was arbitrary, by mollification, this equality holds for any  $\phi \in C_c (\mathbb{R}^n)$. By the Riesz representation theorem, $\mu = \nu$.

Finally, we prove $W_{2,\mathcal{C},\mathbf{M}_t}(\mu, \lambda) \le W_{2,\mathcal{C},\mathbf{M}_t} (\mu,\nu) + W_{2,\mathcal{C},\mathbf{M}_t}(\nu,\lambda)$. We may assume that $\mu \neq \nu,\,\nu \neq \lambda$, and further assume that the right-hand side is finite. Let $\varepsilon>0$ be given. Choose admissible pairs $(\mu_\cdot, v_\cdot)$ from $\mu$ to $\nu$ and $(\nu_\cdot, w_\cdot)$ from $\nu$ to $\lambda$ with $\mu_\cdot$, $\nu_\cdot \subseteq \mathcal{C}$, so that $W^2_{2,\mathcal{C},\mathbf{M}_t} (\mu,\nu) \le \int\!\!\int v_s^\top \mathbf{M}_t v_s d\mu_sds<W^2_{2,\mathcal{C},\mathbf{M}_t} (\mu,\nu) +\varepsilon$ and $W^2_{2,\mathcal{C},\mathbf{M}_t} (\nu,\lambda) \le \int\!\!\int w_s^\top \mathbf{M}_t w_s d\nu_sds < W^2_{2,\mathcal{C},\mathbf{M}_t} (\nu,\lambda)+\varepsilon$, respectively. Concatenate these two admissible pairs by affinely rescaling the parameters onto $[0,a]$ and $[a,1]$, respectively, where $a \in (0,1)$ is a parameter to be chosen. The resulting pair is admissible, because the continuity equation reduces to $\int_{\mathbb{R}^n} \varphi(a,x) d \mu_1 - \int_{\mathbb{R}^n} \varphi(a,x) d\nu_0 =0$ by Lemma~\ref{lemma:boundary_terms} for every $\varphi \in C_c^\infty((0,1)\times\mathbb{R}^n)$. 
Therefore, we have $W^2_{2,\mathcal{C},\mathbf{M}_t} (\mu,\lambda) \le \frac 1 a (W^2_{2,\mathcal{C},\mathbf{M}_t} (\mu,\nu)+\varepsilon) +\frac 1 {1-a} (W^2_{2,\mathcal{C},\mathbf{M}_t} (\nu,\lambda) +\varepsilon)$. Take the limit $\varepsilon \to 0+$ and choose $a =W_{2,\mathcal{C},\mathbf{M}_t}(\mu,\nu)/(W_{2,\mathcal{C},\mathbf{M}_t}(\mu,\nu)+W_{2,\mathcal{C},\mathbf{M}_t}(\nu,\lambda))$. This completes the proof.
\end{proof}

\subsection{Weighted Helmholtz--Hodge Decompositions}
\label{subsec:hhds}
Our contraction conditions will depend on the Helmholtz--Hodge decomposition (HHD)~\cite{BhNo2013}. Here, we develop weighted variants of the decomposition which are suitable for our purposes.
\begin{lemma}[Weighted HHD for curves in $\mathcal{P}_2$]
\label{lemma:vertical_horizontal_decomp}
Let $\mu_\cdot:[0,1]\to\mathcal P_2$ be a narrowly continuous curve, and let $t\ge0$. Define $\boldsymbol\mu(ds,d\mathbf{x})\coloneqq\mu_s(d\mathbf{x})ds$, i.e., for each $A\in\mathcal B([0,1]\times\mathbb R^n)$, $\boldsymbol\mu(A)\coloneqq\int_0^1\mu_s\left(\left\{\mathbf{x}\in\mathbb R^n:(s,\mathbf{x})\in A\right\}\right)ds$. Let $L^2_{\mathbf M_t}(\boldsymbol\mu)$ be the Hilbert space of weighted-square-integrable Borel vector fields (identified up to $\boldsymbol{\mu}$-a.e. equality) with inner product $\langle v,w\rangle_{L^2_{\mathbf M_t}(\boldsymbol\mu)} \coloneqq \int_0^1\int_{\mathbb R^n}v_s(\mathbf{x})^\top\mathbf M_t(\mathbf{x})w_s(\mathbf{x})d\mu_s(\mathbf{x})ds$. Define the subspace of infinitesimal stabilizers of the curve $\mu_\cdot$ by 
\begin{equation}
\begin{split}
    \mathcal V_{\mu_\cdot} (t) \!\coloneqq \!\Biggl\{w\in L^2_{\mathbf M_t}(\boldsymbol\mu)&\!:\! \int_0^1 \!\!\int_{\mathbb R^n} \nabla_{\mathbf{x}}\psi(s,\mathbf{x})\!\cdot\! w_s(\mathbf{x}) d\mu_s(\mathbf{x})ds\\&=0, \text{ }\forall \psi\in C_c^\infty((0,1)\times\mathbb R^n)\Biggr\}
    \end{split}
\end{equation}
and define the subspace of weighted gradient velocities by
\begin{equation}
    \mathcal H _{\mu_\cdot} (t)\coloneqq \overline{\left\{\mathbf M_t(\cdot)^{-1}\nabla_{\mathbf{x}}\psi: \psi\in C_c^\infty((0,1)\times\mathbb R^n)\right\}}^{L^2_{\mathbf M_t}(\boldsymbol\mu)}.
\end{equation}

Then $\mathcal V_{\mu_\cdot} (t)=\mathcal H _{\mu_\cdot} (t)^\perp$, where the orthogonal complement is taken in $L^2_{\mathbf M_t}(\boldsymbol\mu)$. In particular,
\begin{equation}
    L^2_{\mathbf{M}_t}(\boldsymbol{\mu}) = \mathcal{H}_{\mu_\cdot} (t) \oplus \mathcal{V}_{\mu_\cdot} (t).
\end{equation}

Moreover, if the infimum in~\eqref{eq:weighted_action} is finite and $\underline{m} I \preceq \mathbf M_t(\cdot) $ on $\bigcup_{s \in [0,1]}\operatorname{supp}(\mu_s)$ for some $\underline{m}>0$, there is a unique element in $L^2_{\mathbf M_t}(\boldsymbol\mu)$ which attains the weighted action; furthermore, this element is in $\mathcal H _{\mu_\cdot} (t)$.
\end{lemma}
\begin{remark}
We refer to $\mathcal{H}_{\mu_\cdot}(t)$ (resp. $\mathcal{V}_{\mu_\cdot}(t)$) as the weighted horizontal (resp. vertical) space at $\mu_\cdot$ at time $t$.
\end{remark}

\begin{proof}
Below, we write $\mathcal{V} \coloneqq \mathcal{V}_{\mu_\cdot}(t)$ and $\mathcal{H} \coloneqq \mathcal{H}_{\mu_\cdot}(t)$. Let $\mathcal G$ be the collection of Borel sets $B\subseteq\mathbb R^n$ such that $s\mapsto\mu_s(B)$ is Borel. Then $\mathcal G$ is a $\lambda$-system, and it contains the closed sets because $\mu_s(F)=\lim_{k\to\infty}\int \max\{1-kd(\mathbf{x},F),0\}d\mu_s(\mathbf{x})$ for every closed set $F$. Since the closed sets generate $\mathcal B(\mathbb R^n)$, the $\pi$--$\lambda$ theorem gives $\mathcal G=\mathcal B(\mathbb R^n)$. Now, let $\mathcal{G}'$ be the collection of Borel sets $B' \subseteq \mathbb{R}^{1+n}$ such that $s \mapsto \mu_s(B_s)$ is Borel. Then $\mathcal G'$ is a $\lambda$-system, and it contains the measurable rectangles as $\mathcal{G}=\mathcal{B}(\mathbb{R}^n)$. Therefore, another application of the $\pi$--$\lambda$ theorem gives $\mathcal{G}' = \mathcal{B}(\mathbb{R}) \otimes\mathcal{B}(\mathbb{R}^n)=\mathcal{B}(\mathbb{R}^{1+n})$. This, with the monotone convergence theorem, proves that $\boldsymbol{\mu}$ is indeed a well-defined Borel probability measure on $[0,1]\times\mathbb R^n$. Moreover, $L^2_{\mathbf{M}_t}(\boldsymbol{\mu})$ is a well-defined inner product space, and is isometrically isomorphic to $L^2(\boldsymbol{\mu})$. Hence, it is complete.

For $\psi\in C_c^\infty((0,1)\times\mathbb R^n)$, set $h_\psi\coloneqq\mathbf M(t,\cdot)^{-1}\nabla_{\mathbf{x}}\psi$.
Clearly, we have $h_\psi\in L^2_{\mathbf M_t}(\boldsymbol\mu)$. For every $w\in L^2_{\mathbf M_t}(\boldsymbol\mu)$, $\langle w,h_\psi\rangle_{L^2_{\mathbf M_t}(\boldsymbol\mu)} = \int_0^1 \!\! \int_{\mathbb R^n} \nabla_{\mathbf{x}}\psi\cdot w_s d\mu_s(\mathbf{x})ds$. Hence $w\in\mathcal V$ if and only if $w$ is orthogonal to every $h_\psi$. Therefore, $\mathcal V = \{h_\psi:\psi\in C_c^\infty((0,1)\times\mathbb R^n)\}^\perp = \mathcal H^\perp$, and since $\mathcal H$ is a closed subspace, we have $\mathcal V^\perp=(\mathcal H^\perp)^\perp=\mathcal H$.

Now we prove the second assertion. The lower bound assumption on $\mathbf{M}_t$ implies that each $v \in L^2_{\mathbf{M}_t}(\boldsymbol{\mu})$ satisfies $\int_0^1 \left\|v_s\right\|_{L^2(\mu_s)}ds < \infty$. By finiteness of the infimum, there is at least one Borel $v_0$ s.t. $v_0\in L^2_{\mathbf M_t}(\boldsymbol\mu)$ and $ \partial_s\mu_s+\operatorname{div}(\mu_s v_{0,s})=0$ in the weak sense. Another velocity field $v\in L^2_{\mathbf M_t}(\boldsymbol\mu)$ satisfies the same continuity equation with the same curve $s \mapsto \mu_s$ if and only if $v-v_0\in\mathcal V=\mathcal{H}^\perp$. Thus the velocities in $L^2_{\mathbf M_t}(\boldsymbol\mu)$ solving the continuity equation for the fixed curve are precisely the affine subspace $v_0+\mathcal{H}^\perp$. Applying the Pythagorean theorem to the decomposition $L^2_{\mathbf{M}_t} = \mathcal{H} \oplus \mathcal{H}^\perp$ gives a unique element $v_*$ of minimal $\mathbf M_t$-norm in this affine subspace. Moreover, $v_* \perp \mathcal{V}$, so that $v_* \in \mathcal{H}$.
\end{proof}
\begin{remark}
Let $\mu:[0,1]\to\mathcal P_2(\mathbb R^n)$ be narrowly continuous, and let $\boldsymbol\mu(ds,d\mathbf{x})=\mu_s(d\mathbf{x})ds$ be the associated measure on $[0,1]\times\mathbb R^n$. By monotone convergence and the decomposition $f=f^+-f^-$, any Borel measurable $f(s, \mathbf{x}):[0,1]\times\mathbb{R}^n \to \overline {\mathbb{R}}$ gives rise to a Borel measurable $s \mapsto \int_{\mathbb{R}^n} f(s,\mathbf{x})d\mu_s(\mathbf{x})$, where the integral may be left undefined for some $s$. Moreover, we have $ \int_{[0,1]\times\mathbb{R}^n} f(s,\mathbf{x}) d\boldsymbol{\mu}(s,\mathbf{x})=\int_0^1\int_{\mathbb{R}^n} f(s,\mathbf{x}) d \mu_s(\mathbf{x}) ds$: if the former is well-defined, then the latter is also well-defined and the two values agree. Hereafter, we use these facts freely without mentioning it.
\end{remark}

We now state the analogous decomposition for $L^2_{\mathbf{M}_t}(\mu)$ for each $\mu \in \mathcal{P}_2$. Here, $L^2_{\mathbf{M}_t}(\mu)$ is the Hilbert space of Borel vector fields $v$ on $\mathbb{R}^n$ such that $\int v^\top \mathbf{M}_t v d\mu <\infty$, with vector fields identified up to $\mu$-a.s. equality.

Define 
\begin{equation}
\begin{split}
    \mathcal V_{\mu} (t) \!\coloneqq \!\Biggl\{w\in L^2_{\mathbf M_t}(\mu)&:\!\int_{\mathbb R^n} \nabla_{\mathbf{x}}\psi(\mathbf{x})\!\cdot\! w(\mathbf{x}) d\mu(\mathbf{x})\\&=0, \text{ }\forall \psi\in C_c^\infty(\mathbb R^n)\Biggr\}
    \end{split}
\end{equation}
and
\begin{equation}
    \mathcal H_\mu(t) \coloneqq \overline{\{\mathbf M_t^{-1}\nabla_{\mathbf{x}}\psi:\psi\in C_c^\infty(\mathbb R^n)\}}^{L^2_{\mathbf M_t}(\mu)}.
\end{equation}

\begin{lemma}[Weighted HHD for measures in $\mathcal{P}_2$]
For every $\mu\in\mathcal P_2$ and $t\ge0$, we have $\mathcal{V}_\mu(t)=\mathcal{H}_\mu(t)^\perp$ and hence
\begin{equation}
    L^2_{\mathbf M_t}(\mu) = \mathcal H_\mu(t)\oplus \mathcal V_\mu(t).
\end{equation}
For each $w \in L^2_{\mathbf{M}_t} (\mu)$, the orthogonal projection of $w$ onto $\mathcal{H}_\mu(t)$ is the unique minimizer of the kinetic energy functional $v\mapsto \| v\|^2_{L^2_{\mathbf{M}_t}(\mu)}\coloneqq \int_{\mathbb{R}^n} v(\mathbf{x})^\top \mathbf{M}_t(\mathbf{x})v(\mathbf{x})d \mu(\mathbf{x})$ on $L^2_{\mathbf{M}_t}(\mu)$ subject to $\operatorname{div}(v\mu ) = \operatorname{div}(w\mu )$ as distributions in $\mathcal{D}'(\mathbb{R}^n)$.
\end{lemma}

\begin{proof}
The proof is almost identical to that of Lemma~\ref{lemma:vertical_horizontal_decomp}.
\end{proof}

In other words, each choice of the matrix-valued function $\mathbf{M}(t,\cdot)$ gives a canonical isomorphism
\begin{equation}
    L^2_{\mathbf M_t}(\mu)/\ker [v \mapsto \operatorname{div} (v\mu)] \simeq \mathcal H_\mu(t).
\end{equation}

\subsection{The Intrinsic Wasserstein Contraction Theorem}

We are now ready to state a version of the main theorem. We first digress to a version (Theorem~\ref{thm:wasserstein_contraction}) where the intrinsic Wasserstein contraction condition (IWCC)~\eqref{eq:s_integrated_iwcc} is presented in a form which involves integration in $s$. In Section~\ref{sec:iwcc_pointwise}, we present a geometrically clearer version relying on (adequately defined) horizontal tangent cones $\operatorname{HorTan}_\mu(t,\mathcal{C})$ for $\mu \in \mathcal{C}$. This version will be obtained as a consequence of Theorem~\ref{thm:wasserstein_contraction}: see Theorem~\ref{thm:pointwise_iwcc}.

\begin{theorem}[IWCC, $s$-integrated version]
\label{thm:wasserstein_contraction}
Let $\mathcal{C} \subseteq \mathcal{P}_2$ be a nonempty, forward-invariant set under the flow $(\mathsf{S}_{r,t})_{t \ge r}$. Let $\lambda \in \mathbb{R}$, and define $B_\lambda = \dot{\mathbf{M}} + D_{\mathbf{x}}f^\top \mathbf{M} + \mathbf{M} D_{\mathbf{x}} f + 2\lambda \mathbf{M}$ where $\dot{\mathbf{M}}\coloneqq\partial_t \mathbf{M}+D_{\mathbf{x}} \mathbf{M} [f]$ is the material derivative of $\mathbf{M}$ along $\dot{\mathbf{x}} = f(t, \mathbf{x})$. Let $\Omega_{\mathcal{C}}:=\bigcup_{\mu \in \mathcal{C}} \operatorname{supp}(\mu) \subseteq \mathbb{R}^n$, and assume $\kappa_T <\infty$ for all $T>0$ where
\begin{equation}
\label{eq:kappaT_def}
\kappa_T \coloneqq \sup_{\substack{0\le \tau\le T\\\mathbf{x}\in\Omega_{\mathcal C}}}
\lambda_{\max}\left(\mathbf M(\tau,\mathbf{x})^{-1/2} B_\lambda (\tau,\mathbf{x})\mathbf M(\tau,\mathbf{x})^{-1/2}\right).
\end{equation}
Assume $\underline{m} I \preceq \mathbf{M}(t, \mathbf{x}) \preceq \overline {m} I$ for all $(t,\mathbf{x}) \in [0,\infty)\times \Omega_{\mathcal C}$ for some positive constants $\underline{m}, \overline {m}>0$.

Moreover, suppose for every $t\ge0$ and every admissible curve $(\mu_s,v_s)_{s\in[0,1]}$ with $\mu_\cdot\subseteq \mathcal C$ and $v\in \mathcal{H}_{\mu_\cdot}(t)$, we have the following \emph{$s$-integrated IWCC}:
\begin{equation} 
\label{eq:s_integrated_iwcc}
\int_0^1 \!\!\int_{\mathbb{R}^n} v_s (\mathbf{x})^\top  B_\lambda(t, \mathbf{x})  v_s (\mathbf{x}) d \mu_s(\mathbf{x}) d s \le 0.
\end{equation}

Then, the exponential bound
\begin{equation}
\begin{split}
    W_2(\mathsf{S}_{r,t} \mu, \mathsf{S}_{r,t}\nu) &\le W_{2,\mathcal{C}} (\mathsf{S}_{r,t} \mu, \mathsf{S}_{r,t}\nu)\\& \le \sqrt{\frac{\overline{m}}{\underline{m}}} e^{-\lambda (t-r)} W_{2,\mathcal{C}}(\mu, \nu).
\end{split}
\end{equation}
holds for all $\mu, \nu \in \mathcal{C}$ and $t \ge r \ge 0$. 
\end{theorem}
Note that $\lambda \in \mathbb{R}$ could be any real number. 

\begin{remark}
The finiteness assumption for $\kappa_T$ is automatically satisfied if we have $B_\lambda \preceq 0$ pointwise as assumed in \cite{LoSl1998,BoSl2019}, or if $\Omega_{\mathcal{C}}$ is bounded. It means that the infinitesimal energy weighted by $\mathbf{M}$ cannot grow arbitrarily fast along the flow: see~\eqref{eq:energy_density_growth}.
\end{remark}

We postpone the proof of Theorem~\ref{thm:wasserstein_contraction} and first state and prove some useful lemmas.

\begin{lemma}
\label{lemma:pushforward_admissible}
Let $(\mu_s,v_s)_{s\in[0,1]}$ be an admissible pair, and let $\Phi:[0,1]\times\mathbb R^n\to\mathbb R^n$ be a $C^1$ map such that each $\Phi_s\coloneqq\Phi(s,\cdot)$ is an injection. Suppose $\nu_s\coloneqq(\Phi_s)_\#\mu_s\in\mathcal P_2$ for all $s\in[0,1]$ and assume
\begin{equation}
\label{eq:integrability_condition_pushforward_lemma}
    \int_0^1 \left( \int_{\mathbb R^n} \left|D_{\mathbf{x}}\Phi_s(\mathbf{x})v_s(\mathbf{x}) + \partial_s\Phi_s(\mathbf{x})\right|^2d\mu_s(\mathbf{x})\right)^{1/2}ds<\infty.
\end{equation}
Define $w_s$ on $\mathbb R^n$ by
\begin{equation}
\label{eq:pushforwarded_velocity_field}
    w_s(\mathbf{y}) = \left[ D_{\mathbf{x}}\Phi_s(\mathbf{x})v_s(\mathbf{x}) + \partial_s\Phi_s(\mathbf{x}) \right]_{\mathbf{x}=\Phi_s^{-1}(\mathbf{y})} \mathbf 1_{\Phi_s(\mathbb R^n)}(\mathbf{y}).
\end{equation}
Then $(\nu_s,w_s)_{s\in[0,1]}$ is admissible.
\end{lemma}

\begin{proof}
Let $f \in C_b(\mathbb{R}^n)$. Then, we have $\int fd\nu_s= \int f \circ \Phi_s d\mu_s$. We prove the latter is continuous in $s$. Let $s_k \to s \in [0,1]$ and let $\varepsilon>0$ be given. By Prokhorov's theorem, there exists a compact $K \subseteq \mathbb{R}^n$ such that $\mu_s (K^C) < \varepsilon/(2\|f\|_{\sup} + 1)$ for all $s \in [0,1]$.
We have
$\left| \int f\circ \Phi_{s_k} d \mu_{s_k} - \int f \circ \Phi_s d \mu_s \right| \le \int | f \circ \Phi_{s_k} - f \circ \Phi_s | d \mu_{s_k} + \left| \int f\circ \Phi_s d \mu_{s_k} - \int f \circ \Phi_s d \mu_s\right|$.
The first term can be decomposed as $\int_K + \int_{K^c}$. For sufficiently large $k$, $\int _K < \varepsilon$ because $f \circ \Phi$ is uniformly continuous on the compact set $[0,1] \times K$, and $\int_{K^C} < \varepsilon$ because $\mu_{s_k} (K^C) < \varepsilon/(2\|f\|_{\sup} + 1)$. The second term is less than $\varepsilon$ for sufficiently large $k$ because $f \circ \Phi_s \in C_b(\mathbb{R}^n)$ and $\mu_\cdot$ is narrowly continuous. This proves that $(\nu_s)_{s \in [0,1]}$ is narrowly continuous.

Since $[0,1] \times \mathbb R^n$ is Polish and $T:(s,\mathbf{x}) \mapsto (s, \Phi_s(\mathbf{x}))$ is a Borel-measurable injection, the Lusin--Souslin theorem~\cite[Corollary~15.2]{Ke1995} implies that $T([0,1] \times \mathbb R^n)$ is Borel and that $T$ is a Borel isomorphism onto its image. This proves that $w$ is Borel measurable on $[0,1]\times\mathbb R^n$. Moreover, 
\begin{equation}
\begin{split}
    \int_0^1 \| w_s\|_{L^2(\nu_s)} ds  =  \int_0^1 \Biggl(&\int_{\mathbb{R}^n} |D_{\mathbf{x}}\Phi_s (\mathbf{x})v_s(\mathbf{x})\\&+ \partial_s \Phi _s (\mathbf{x}) |^2 d\mu_s(\mathbf{x}) \Biggr)^{1/2}ds <\infty.
    \end{split}
\end{equation}

It remains to check the continuity equation. First, we note that the weak formulation for $(\mu_s,v_s)$ extends to test functions in $C_c^1((0,1)\times\mathbb R^n)$. Indeed, if $\varphi\in C_c^1((0,1)\times\mathbb R^n)$, then the mollifications $\varphi^\epsilon=\eta_\epsilon*\varphi$ by the standard mollifier $\eta$ approximate $\varphi$ uniformly, and their first partial derivatives also uniformly converge to those of $\varphi$~\cite{Ev2010}. Passing to the limit in the weak formulation is justified by $\int_0^1\|v_s\|_{L^1(\mu_s)}ds<\infty$.

Now let $\psi\in C_c^\infty((0,1)\times\mathbb R^n)$ and define $\varphi(s,\mathbf{x}):=\psi(s,\Phi_s(\mathbf{x}))=\psi(T(s,\mathbf{x}))$. We have $\varphi \in C^1([0,1]\times\mathbb R^n)$ and the support of $\varphi$ is contained in $[\delta, 1-\delta]\times \mathbb{R}^n$ for sufficiently small $\delta >0$. Therefore, if we define $\varphi_R (s,\mathbf{x})= \varphi(s,\mathbf{x}) \chi(\mathbf{x}/R)$, where we fix any $\chi \in C_c^\infty(\mathbb{R}^{n})$ such that $0 \le \chi \le 1$, $\chi = 1$ on $\|\mathbf{x}\| \le 1$, $\chi = 0$ on $\|\mathbf{x}\| \ge 2$, then we have $\varphi_R \in C_c^1((0,1)\times\mathbb{R}^n)$, so it may be used as a test function by the preceding mollification argument. In the corresponding weak formulation, the $\nabla \chi$ term vanishes as $R \to \infty$ because it is of order $\mathcal{O}(R^{-1})$. The remaining terms converge by dominated convergence, where the domination is justified by \eqref{eq:integrability_condition_pushforward_lemma}. Thus,
\begin{equation*}
\begin{split}
    0&=\int_0^1 \!\! \int_{\mathbb R^n}\partial_s\psi(s,\Phi_s(\mathbf{x})) +\nabla \psi(s,\Phi_s(\mathbf{x}))\cdot (D_{\mathbf{x}}\Phi_s(\mathbf{x})v_s(\mathbf{x})\\
    &\qquad\qquad\qquad\qquad+\partial_s\Phi_s(\mathbf{x})) d\mu_s(\mathbf{x})ds\\
    &= \int_0^1 \!\! \int_{\mathbb R^n}  \partial_s\psi(s,\mathbf{y}) +\nabla\psi(s,\mathbf{y})\cdot w_s(\mathbf{y}) d\nu_s(\mathbf{y})ds,
\end{split}
\end{equation*}
where the second equality uses $\nu_s=(\Phi_s)_\# \mu_s$ and $w_s(\Phi_s(\mathbf{x}))=D_{\mathbf{x}}\Phi_s(\mathbf{x})v_s(\mathbf{x})+\partial_s \Phi_s(\mathbf{x})$. This completes the proof.
\end{proof}

Finally, for the sake of completeness, we write down a variant of the comparison lemma:
\begin{lemma}
\label{lemma:comparison_dini}
Let $A:[a,b] \to \mathbb{R}$ be lower semicontinuous from the left: for all $t \in (a,b]$, $A(t) \le \liminf_{s \to t-} A(s)$. If for some $\rho \in \mathbb{R}$ we have $D^+ A(t)\coloneqq \limsup _{h \to 0+} \frac{A(t+h)-A(t)}{h} \le - \rho A(t)$ for all $t\in[a,b)$, then we have $A(t) \le e^{-\rho (t-a)}A(a)$ for all $t \in [a,b]$.
\end{lemma}
\begin{proof}
    The statement follows from applying \cite[Lemma~2.3]{WaZh2011} to the function $B(t)\coloneqq g(t)A(t)$ where $g(t)\coloneqq e^{\rho(t-a)}$ to conclude that $B$ is non-increasing.
\end{proof}

Now we are ready to prove Theorem~\ref{thm:wasserstein_contraction}.

\begin{proof}
Note that $\underline m I \preceq \mathbf{M} \preceq \overline m I$ implies, for every $\alpha, \beta \in \mathcal{C}$ and every $\tau \ge 0$,
\begin{equation}
    \underline m W^2_{2,\mathcal{C}}  (\alpha, \beta) \le W^2_{2,\mathcal{C},\mathbf{M}_\tau}(\alpha, \beta) \le \overline m W^2_{2,\mathcal{C}}(\alpha, \beta).
\end{equation}
Therefore, it suffices to prove the weighted estimate \begin{equation}
W^2_{2,\mathcal{C},\mathbf M_t}(\mathsf S_{r,t} \mu_0, \mathsf S_{r,t} \mu_1) \le e^{-2 \lambda (t-r)} W^2_{2,\mathcal {C}, \mathbf{M}_r} (\mu_0, \mu_1),
\end{equation}
for all $\mu_0, \mu_1 \in \mathcal{C}$ and $0 \le r< t$. Fix $0 \le r<t<T$, and fix a narrowly continuous curve $(\mu_s)_{s \in [0,1]}$ from $\mu_0$ to $\mu_1$ with values in $\mathcal{C}$. For all $\tau \in [r,t]$, set $\mu^\tau_s = \mathsf S _{r,\tau} \mu_s \in \mathcal{C}$ and define $A(\tau) = \mathcal{A}_{\mathbf{M}_\tau} (\mu^\tau_\cdot)$. We will apply Lemma~\ref{lemma:comparison_dini} to the function $A$, which will give us the bound $A(t)\le e^{-2\lambda(t-r)}A(r)$. Note that we may assume $A(r)<\infty$.

Hereafter, we check the conditions of Lemma~\ref{lemma:comparison_dini}. We first prove an auxiliary estimate for the flow. Since $\mathcal C$ is forward invariant and $\Phi_{\tau, \tau+h}$ is continuous, $\Phi_{\tau,\tau+h}(\Omega_{\mathcal C})\subseteq\Omega_{\mathcal C}$ whenever $0\le\tau\le\tau+h\le T$.
For any fixed $\mathbf{x}\in\Omega_{\mathcal C}$ and $\xi\in\mathbb R^n$, using the variational equation~\cite[Theorem~V.3.1]{Ha2002} gives:
\begin{equation}
\frac{d}{dh}\!\!\left[\xi_h^\top\mathbf M(\tau_h,\Phi_{\tau,\tau_h}(\mathbf{x}))\xi_h\right]\!\!=\! \xi_h ^\top \!\!\left[B_\lambda-2\lambda\mathbf M \right]\!(\tau_h,\Phi_{\tau,\tau_h}(\mathbf{x})) \xi_h ,
\end{equation}
where $\tau_h \coloneqq \tau+h$ and $\xi_h \coloneqq D\Phi_{\tau,\tau_h}(\mathbf{x})\xi$. By the definition of $\kappa_T$, the right-hand side is at most $(\kappa_T-2\lambda)$ times the expression being differentiated. Hence, with $\gamma_T\coloneqq \kappa_T-2\lambda$, we have
\begin{equation}
\label{eq:energy_density_growth}
\xi_h^\top \mathbf M(\tau_h,\Phi_{\tau,\tau_h}(\mathbf{x})) \xi_h \le e^{\gamma_T h}\xi^\top\mathbf M(\tau,\mathbf{x})\xi.
\end{equation}

We next verify that $A$ is finite and lower semicontinuous from the left. Let $\tau\in(r,t]$. For each $0<h\le\tau-r$, if $A(\tau-h)<\infty$, then the last statement of Lemma~\ref{lemma:vertical_horizontal_decomp} applies because $\bigcup_{s \in [0,1]}\operatorname{supp}\mu_s^{\tau-h}\subseteq\Omega_{\mathcal C}$. We let $u^{\tau-h}$ be the minimizing velocity for $A(\tau-h)$. Pushing the admissible pair $(\mu^{\tau-h}_s, u^{\tau-h}_s)_{s \in [0,1]}$ forward by $\Phi_{\tau-h,\tau}$ (here we use Lemma~\ref{lemma:pushforward_admissible}; the integrability condition is satisfied because of \eqref{eq:energy_density_growth} and $A(\tau -h)<\infty$) and using \eqref{eq:energy_density_growth} gives
\begin{equation}
A(\tau)\le e^{\gamma_T h}A(\tau-h),
\end{equation}
which trivially holds in the case of $A(\tau-h)=\infty$. Taking $h=\tau-r$ proves $A(\tau)<\infty$ and hence $A<\infty$ everywhere. Moreover, letting $h\to0+$ gives $A(\tau)\le\liminf_{h\to0+}A(\tau-h)$, so $A$ is lower semicontinuous from the left.

It remains to prove the Dini derivative inequality. Let $\tau\in[r,t)$, and let $u^\tau \in \mathcal{H}_{\mu_\cdot^\tau}(\tau)$ be the minimizing velocity for $A(\tau)$. For $0<h\le t-\tau$, pushing the pair $(\mu_s^\tau,u_s^\tau)_{s\in[0,1]}$ forward by $\Phi_{\tau,\tau+h}$ using Lemma~\ref{lemma:pushforward_admissible} gives, by the definition of $A(\tau+h)$,
\begin{equation}
\begin{split}
\frac{A(\tau_h)-A(\tau)}{h} \!\le\! \int_0^1 \!\!\int_{\mathbb R^n} &\frac{1}{h} \Big[ \left(D\Phi_{\tau,\tau_h}(\mathbf{x})u_s^\tau(\mathbf{x})\right)^\top\\& \mathbf M(\tau_h,\Phi_{\tau,\tau_h}(\mathbf{x})) \! \left(D\Phi_{\tau,\tau_h}(\mathbf{x})u_s^\tau(\mathbf{x})\right)\\
&- u_s^\tau(\mathbf{x})^\top\mathbf M(\tau,\mathbf{x})u_s^\tau(\mathbf{x})\Big] d\mu_s^\tau(\mathbf{x})ds,
\end{split}
\end{equation}
where $\tau_h \coloneqq \tau+h$.

For each $x \in \Omega_{\mathcal{C}}$, the integrand is no greater than $\left|\frac{e^{\gamma_T h}-1} {h} \right|u^\tau_s(\mathbf{x})^\top \mathbf{M}(\tau, \mathbf{x})u^\tau_s(\mathbf{x})$ by \eqref{eq:energy_density_growth}. Since $A(\tau)<\infty$ and $\left|\frac{e^{\gamma_T h}-1}{h}\right| \le |\gamma_T| e^{|\gamma_T| T}$, the reverse Fatou lemma is applicable for any $(h_n)_{n \ge 1}$ strictly decreasing to $0$. This gives:
\begin{equation}
\begin{aligned}
D^+A(\tau) &\le \int_0^1\!\!\int_{\mathbb R^n} u_s^\tau(\mathbf{x})^\top
\left[ B_\lambda-2\lambda\mathbf M \right] u_s^\tau(\mathbf{x}) d\mu_s^\tau(\mathbf{x})ds\\
&=\int_0^1\!\!\int_{\mathbb R^n} u_s^\tau(\mathbf{x})^\top B_\lambda u_s^\tau(\mathbf{x}) d\mu_s^\tau(\mathbf{x})ds - 2\lambda A(\tau)\\
&\le -2 \lambda A(\tau),
\end{aligned}
\end{equation}
where $B_\lambda = B_\lambda(\tau,\mathbf{x})$, $\mathbf{M}=\mathbf{M}(\tau,\mathbf{x})$, and the last inequality comes from the IWCC. Therefore, Lemma~\ref{lemma:comparison_dini} can be applied to conclude $A(t)\le e^{-2\lambda(t-r)}A(r)$.

Since $\mu_\cdot^t$ joins $\mathsf S_{r,t}\mu_0$ to $\mathsf S_{r,t}\mu_1$ inside $\mathcal C$, we get $W^2_{2,\mathcal C,\mathbf M_t} (\mathsf S_{r,t}\mu_0,\mathsf S_{r,t}\mu_1) \le e^{-2\lambda(t-r)} \mathcal A_{\mathbf M_r}(\mu_\cdot)$.
Taking the infimum over narrowly continuous curves $\mu_\cdot$ in $\mathcal C$ from $\mu_0$ to $\mu_1$ proves the weighted estimate. This concludes the proof of the theorem.
\end{proof}
\begin{remark}
    Even without the condition $\mathbf{M} \preceq \overline{m}I$, we still have the exponential upper bound $W_{2}(\mathsf{S}_{r,t}\mu_0,\mathsf{S}_{r,t} \mu_1) \le W_{2,\mathcal{C}}(\mathsf{S}_{r,t}\mu_0,\mathsf{S}_{r,t} \mu_1) \le \frac{1}{\sqrt{\underline{m}}} e^{-\lambda(t-r)}W_{2,\mathcal{C},\mathbf{M}_r}(\mu_0,\mu_1)$.
\end{remark}

Finally, we record a gauge invariance inherent to our setup, which is related to the weighted Helmholtz--Hodge decompositions in Section~\ref{subsec:hhds}.

\begin{proposition}[Vertical gauge invariance]
Let $f=g+z$ for some functions $g$ and $z$, where  $g$ and $D_{\mathbf{x}}g$ are continuous. Let $\mathcal{C} \subseteq \mathcal{P}_2$ be a forward-invariant subset under the flow $(\mathsf{S}^f_{r,s})_{s \ge r}$. Suppose $\int_{\mathbb{R}^n} \nabla \psi \cdot z_t d\mu = 0 $ for all $t \ge 0$, $\psi \in C_c^\infty(\mathbb{R}^n)$, and $\mu \in \mathcal{C}$. Also, suppose for each $T>0$ there exists some $L_T > 0$ such that
\begin{equation}
    |f(t,\mathbf{x})|+ |g(t,\mathbf{x})| \le L_T (1+|\mathbf{x}|)
\end{equation}
on $[0,T]\times\mathbb{R}^n$. (This implies $g$ is forward complete by Gr\"{o}nwall's inequality.)

Then $f$ and $g$ induce the same pushforward maps on $\mathcal{C}$:
\begin{equation}
    \mathsf{S}^f_{r,s}\mu = \mathsf{S}^g_{r,s}\mu,
\end{equation}
for all $\mu \in \mathcal{C}$ and $s \ge r \ge 0$. In particular, $\mathcal{C}$ is forward invariant under $\mathsf{S}^g_{r,s}$. Moreover, if $g$ satisfies the conditions of Theorem~\ref{thm:wasserstein_contraction}, then the conclusion of that theorem also holds for $\mathsf{S}^f_{r,s}$, $s \ge r \ge 0$.
\end{proposition}

\begin{proof}
Let $\mu \in \mathcal{C}\subseteq \mathcal{P}_2$ and $s \ge r \ge 0$ be given. Fix any $T>s$. 
Below, the symbols $h$, $h_1$, $h_2$ mean either $f$ or $g$. Define $\mu_t^h \coloneqq \mathsf{S}^h_{r,t} \mu$ on $t \in [r,T]$.

By dominated convergence, $t \mapsto \mu^h_t$, $h \in \{f,g\}$ are narrowly continuous. Moreover, any solution to $\dot{\mathbf{x}}= h(t,\mathbf{x})$ on $[r,T]$ satisfies $|\mathbf{x}(t)| \le (|\mathbf{x}(r)|+1)e^{L_T T}$ for all $t \in [r,T]$ and all initial conditions $\mathbf{x}(r)\in\mathbb{R}^n$ by Gr\"{o}nwall's inequality.

Also, \begin{equation} \label{eq:lemma816_cond}
\int_r^T\!\!\!\! \int_{\mathbb{R}^n} \! | h_1(s,\mathbf{x})|d\mu_s^{h_2} (\mathbf{x})ds =\!\! \int_r^T \!\!\!\!\int_{\mathbb{R}^n} \! |h_1(s,\Phi_{r,s}^{h_2}(\mathbf{x}))| d \mu (\mathbf{x})ds.
\end{equation}
Because the integrand is bounded above by $L_T \left(1+(|\mathbf{x}|+1)e^{L_T T} \right)$ and we have $\int |\mathbf{x}| d \mu (\mathbf{x}) \le \left(\int |\mathbf{x}|^2 d \mu (\mathbf{x}) \right)^{1/2}  <\infty$, the whole integral in \eqref{eq:lemma816_cond} is finite. Therefore, we could apply \cite[Lemma 8.1.6]{AmGi2008} to conclude that $t \mapsto \mu^h_t$ is a weak solution of the continuity equation $\partial_t \mu^h_t + \nabla \cdot (h_t \mu_t)=0$ on $(r,T) \times \mathbb{R}^n$. Moreover, $\mu^f_\cdot$ is also a weak solution of $\partial_t \mu_t + \nabla \cdot (g_t \mu_t)=0$ because we have $\int_{\mathbb{R}^n} \nabla \psi \cdot (f_t - g_t) d \mu^f_t  =0$ for all $t \in [r,T]$ and all $\psi \in C_c^\infty (\mathbb{R}^n)$. Therefore, the narrowly continuous curves of finite signed measures $t \mapsto \pm(\mu^f_t - \mu^g_t)$ weakly solve $\partial_t \mu_t + \nabla \cdot (g_t\mu_t) = 0$.

Here, note that $\int_r^T \!\!\int_{\mathbb{R}^n} \left|g_s(\mathbf{x})\right| \,d\!\left|\mu^f_s-\mu^g_s\right|\!(\mathbf{x})ds < \infty$ because $\left|\mu^f_s-\mu^g_s\right| \le \mu^f_s +\mu^g_s$. The outer integrand is measurable in $s$ because by the Riesz representation theorem, $s \mapsto |\sigma_s|  (U)$ is lower semicontinuous for each open $U \subseteq \mathbb{R}^n$ and each narrowly continuous curve of finite signed measures $s\mapsto \sigma_s$ on $(\mathbb{R}^n, \mathcal{B}(\mathbb{R}^n))$.

Since $\mu^f_r = \mu^g_r$, the uniqueness theorem for the continuity equation~\cite[Proposition 8.1.7]{AmGi2008} gives us $\mu^f_\cdot = \mu^g_\cdot$. In particular, we have $\mathsf{S}^f_{r,s}\mu = \mathsf{S}^g_{r,s}\mu$. This completes the proof of the proposition.
\end{proof}

\section{IWCC in Pointwise Form}
\label{sec:iwcc_pointwise}
Theorem~\ref{thm:wasserstein_contraction} is formulated using infinitesimal stabilizers of the whole curve, and the relevant Hilbert space was $L^2_{\mathbf{M}_t}(\boldsymbol{\mu})$, with $\boldsymbol{\mu}(ds,d\mathbf{x}) = \mu_s(d\mathbf{x})ds$. In this section, we derive a pointwise sufficient condition for the $s$-integrated IWCC~\eqref{eq:s_integrated_iwcc}. As expected, this sufficient condition will be stated on a suitably defined tangent cone for each $\mu \in \mathcal{C}$, hence contracting infinitesimal curve lengths which are tangent to $\mathcal{C}$.

Recall that if $K\subseteq X$ is a subset of a normed vector space and $x \in K$, the (Bouligand) contingent cone~\cite{AuFr2009} is defined by $T_K(x) \coloneqq \left\{ v \in X : \liminf_{h \to 0+} d (x+hv, K)/h = 0 \right\}$.
We consider the analogous object in the 2-Wasserstein space. The following definition of the Wasserstein contingent cone previously appeared in \cite{BoFr2024}.

\begin{definition}[Wasserstein contingent cone]
Let $\mathcal{C} \subseteq \mathcal{P}_2$. The Wasserstein contingent cone $\operatorname{Tan}_\mu (\mathcal{C})$ at $\mu \in \mathcal{C}$ is defined as $\operatorname{Tan}_\mu (\mathcal{C})\!\coloneqq\! \left\{ \!v\! \in\! L^2(\mathbb{R}^n,\mathbb{R}^n;\mu)\!:\! \liminf_{h \to 0+} \! W_2\big((\mathrm{Id}+hv)_{\#}\mu,\mathcal{C} \big) / h \!=\! 0 \!\right\}\!$.
\end{definition}
This is well-defined because $(\mathrm{Id}+hv)_{\#}\mu$ does not depend on the choice of the representative $v \in L^2(\mathbb{R}^n,\mathbb{R}^n;\mu)$.

Now we define the tangent objects which will play a key role in our pointwise form of the IWCC.
\begin{definition}[Weighted horizontal tangent cone]
\label{def:weighted_horizontal_tangent_cone}
The weighted horizontal tangent cone \emph{(whtc)} of $\mathcal{C}$ at $\mu \in \mathcal{C}$ (and at time $t$) is defined as:
\begin{equation}
\label{eq:def_hortan}    \operatorname{HorTan}_\mu(t,\mathcal{C})\coloneqq  \mathcal{H}_\mu(t) \cap \operatorname{Tan}_\mu(\mathcal{C}),
\end{equation}
where $\mathcal{H}_\mu(t)\coloneqq \overline{\left\{\mathbf{M}_t^{-1} \nabla_{\mathbf{x}} \psi : \psi \in C_c^\infty(\mathbb{R}^n)\right\}}^{L^2_{\mathbf{M}_t}(\mu)}$.
\end{definition}
Note that both $\mathcal{H}_\mu(t)$ and $\operatorname{Tan}_\mu(\mathcal{C})$ consist of Borel-measurable vector fields identified up to $\mu$-a.s. equality, and hence their intersection, $\operatorname{HorTan}_\mu(t,\mathcal{C})$, also consists of such elements.

The following lemma is closely related to \cite[Proposition~8.4.6]{AmGi2008} but does not require $\|v_s\|_{L^2(\mu_s)}\le |\mu'|(s)$ to hold $s$-a.e.
\begin{lemma}[Euler approximation for admissible pairs]
\label{lemma:euler_approx}
    Let $(\mu_s,v_s)_{s\in[0,1]}$ be an admissible pair. Then, we have 
    \begin{equation}
    \label{eq:euler_type_approx}
        \lim_{h \to 0+} \frac{1}{h}W_2\big(\mu_{s+h}, (\mathrm{Id}+h v_{s})_{\#}\mu_s\big)  = 0
    \end{equation}
    for a.e. $s \in [0,1]$.
    If moreover $\mu_s\in\mathcal C$ for all $s\in[0,1]$, we have $v_s \in \operatorname{Tan}_{\mu_s}(\mathcal{C})$ for a.e. $s \in [0,1]$.
\end{lemma}
\begin{proof}
By \cite[Theorem~3.4]{AmCr2014}, there exists a probability measure $\eta \in \mathcal{P}(\mathbb{R}^n \times \Gamma)$ where $\Gamma = C([0,1];\mathbb{R}^n)$, such that $\mu_s = (e_s \circ \pi_{\Gamma})_{\#} \eta$ for all $s \in [0,1]$ and for $\eta$-a.s. $(\mathbf{x}_0,\gamma)$, $\gamma$ is a solution of the ODE $\dot {\mathbf{x}}(s) = v_s(\mathbf{x}(s))$ with $\gamma(0)=\mathbf{x}_0$, where the ODE is understood as its integral form. Here, $e_s:\Gamma \to \mathbb{R}^n$ is evaluation at $s\in[0,1]$, and $\pi_\Gamma: \mathbb{R}^n \times \Gamma \to \Gamma$ is the projection map. Note that both maps are measurable and also that $\Gamma$ is a Polish metric space.

Let us define an auxiliary function $G_s(\mathbf{x}_0,\gamma)=G(s,\mathbf{x}_0,\gamma) \coloneqq  v(s,\gamma(s))$. Since the map $(s,\gamma) \mapsto \gamma(s)$ is (jointly) continuous, $G$ is measurable. Moreover, for all $s \in [0,1]$,
$\|G_s\|^2_{L^2(\eta)} = \int |v(s,\gamma(s)|^2 d  \eta (\mathbf{x},\gamma) = \int | v(s, \mathbf{y})|^2 d\mu_s (\mathbf{y}) = \|v_s\|^2_{L^2(\mu_s)}$.
Therefore, $s \mapsto \|G_s\|_{L^2(\eta)}$ is in $L^1(ds)$, and therefore is finite almost everywhere.

We also define the map $\tilde G: \mathbb{R} \to L^2(\eta)$ by $\tilde G(s)=G_s$ whenever $\|G_s\|_{L^2(\eta)}<\infty$ and otherwise $\tilde G(s)\coloneqq0$. Now we prove that $\tilde G$ is strongly $ds$-measurable. First, we note that $L^2(\eta) \coloneqq L^2(\mathbb{R}^n \times \Gamma, \mathbb{R}^n ; \eta)$ is separable; this follows from the second-countability of $\mathbb{R}^n \times \Gamma$: see \cite[Proposition 3.4.5]{Co2013}. Therefore, it is enough to prove weak $ds$-measurability by \cite[Theorem 1.1.20]{Hyva2016}; this follows from the identification $L^2(\eta)^* \simeq L^2(\eta)$.

As we have $\int_{\mathbb{R}} \| \tilde G(s) \| ds = \int_0^1 \|G_s\|_{L^2(\eta)} <\infty$, $\tilde G$ is Bochner integrable. Therefore, by the Lebesgue differentiation theorem for Banach-valued functions~\cite[Corollary 2.3.5]{Hyva2016}, we have
\begin{equation}
    \left\| \frac 1 h \int_s^{s+h} \tilde G(r) d r - \tilde G (s) \right\|_{L^2(\eta)} \to 0 \text{ as }h \to 0+ 
\end{equation}
for a.e. $s \in [0,1)$. Fix such $s$. We now prove~\eqref{eq:euler_type_approx} for such $s$. Note that we may assume $\tilde G(s) = G_s$.

Fix any $h>0$ such that $s+h \in [0,1]$. Consider the following (measurable) map from $\mathbb{R}^n \times \Gamma$ to $\mathbb{R}^n \times \mathbb{R}^n$: $(\mathbf{x}_0, \gamma) \mapsto \big(\gamma(s+h), \gamma(s)+h v(s,\gamma(s)) \big)$. The pushforward of $\eta$ with respect to this map is a valid transport plan from $\mu_{s+h}$ to $(\mathrm{Id}+hv_s)_{\#}\mu_s$.
If we define $B_h \coloneqq \int_s^{s+h} \tilde G(r)dr \in L^2(\eta)$, then by applying \cite[Proposition 1.2.25]{Hyva2016} with $X=\mathbb{R}^n$ we have:
\begin{equation}
    B_h(\mathbf{x}_0,\gamma) = \int_s^{s+h} G(r,\mathbf{x}_0,\gamma) dr = \int_s^{s+h} v_r(\gamma(r))dr  = \Delta \gamma
\end{equation}
for $\eta$-a.s. $(\mathbf{x}_0,\gamma)$, where $\Delta \gamma \coloneqq \gamma(s+h)-\gamma(s)$.
Now, we have 
\begin{equation}
\begin{split}
    W_2^2 \big( \mu_{s+h},\,\, &(\mathrm{Id}+h v_s)_{\#}\mu_s\big)
    \le \int \left| \Delta \gamma -h v_s(\gamma(s)) \right|^2 d\eta \\
    &=\int \left| B_h(\mathbf{x}_0,\gamma)-h G_s(\mathbf{x}_0,\gamma)\right|^2 d \eta \\&= h^2 \left\|\frac 1 h \int_s^{s+h}\tilde G(r)dr - \tilde G(s)\right\|_{L^2(\eta)}^2.
\end{split}
\end{equation}
This completes the proof of the first part of the lemma. The second part follows immediately from the definition of $\operatorname{Tan}_{\mu_s} (\mathcal{C})$ because we have $\mu_{s+h} \in \mathcal{C}$.
\end{proof}

The following lemma relates the curve-wise weighted horizontal space $\mathcal{H}_{\mu_\cdot}(t)$ to the pointwise weighted horizontal tangent cones $\operatorname{HorTan}_{\mu_s}(t,\mathcal{C})$, $0 \le s \le 1$.
\begin{lemma}
\label{lemma:horizontal_slicing}
Let $t\ge0$, and let $(\mu_s,v_s)_{s\in[0,1]}$ be an admissible pair with $\mu_s\in\mathcal C$ for every $s\in[0,1]$. If $v\in\mathcal H_{\mu_\cdot}(t)$, then
\begin{equation*}
    v_s\in \operatorname{HorTan}_{\mu_s}(t,\mathcal C) \text{ for a.e. }s\in[0,1].
\end{equation*}
\end{lemma}

\begin{proof}
We prove the two inclusions in the definition of $\operatorname{HorTan}_{\mu_s}(t,\mathcal C)$. Since $v\in\mathcal H_{\mu_\cdot}(t)$, we have $v \in L^2_{\mathbf{M}_t}(\boldsymbol{\mu})$, and there exist functions $\psi_k\in C_c^\infty((0,1)\times\mathbb R^n)$ such that $h_k(s,\mathbf{x}) \coloneqq \mathbf M_t(\mathbf{x})^{-1}\nabla_{\mathbf{x}}\psi_k(s,\mathbf{x})$ satisfies
\begin{equation*}
\int_0^1\!\!\int_{\mathbb R^n} (v_s-h_{k,s})^\top\mathbf M_t(v_s-h_{k,s}) d\mu_sds \to 0.
\end{equation*}
The inner integrals, viewed as functions of $s$, converge to zero in $L^1(ds)$. Therefore, after passing to a subsequence, they converge to zero $s$-a.e. For such $s$, we have $h_{k}(s,\cdot)\in\{\mathbf M_t^{-1}\nabla\phi:\phi\in C_c^\infty(\mathbb R^n)\}$ and therefore the preceding convergence gives $v_s\in\mathcal H_{\mu_s}(t)$.

It remains to show $v_s\in\operatorname{Tan}_{\mu_s}(\mathcal C)$ for a.e. $s\in[0,1]$. This was achieved by Lemma~\ref{lemma:euler_approx}.
\end{proof}

\begin{theorem}[IWCC, pointwise version]
\label{thm:pointwise_iwcc}
Suppose that for every $t \ge 0$ and $\mu\in\mathcal C$, we have the following (pointwise) IWCC at $(t,\mu)$:
\begin{equation}
\label{eq:pointwise_iwcc}
    \int_{\mathbb R^n} \! v(\mathbf{x})^\top B_\lambda(t,\mathbf{x})v(\mathbf{x}) d\mu(\mathbf{x}) \le0, \text{ }\forall v\in\operatorname{HorTan}_\mu(t,\mathcal C).
\end{equation}

Then, the $s$-integrated IWCC~\eqref{eq:s_integrated_iwcc} is satisfied. In particular, in Theorem~\ref{thm:wasserstein_contraction}, we can replace the condition \eqref{eq:s_integrated_iwcc} by \eqref{eq:pointwise_iwcc}.
\end{theorem}

\begin{proof}
Let $t \ge 0$ and let $(\mu_s,v_s)_{s\in[0,1]}$ be an admissible pair with $\mu_\cdot\subseteq\mathcal C$ and $v\in\mathcal H_{\mu_\cdot}(t)$.
By Lemma~\ref{lemma:horizontal_slicing}, $v_s\in\operatorname{HorTan}_{\mu_s}(t,\mathcal C)$ for a.e. $s\in[0,1]$. Applying \eqref{eq:pointwise_iwcc} to $(\mu_s,v_s)$ and integrating in $s$ gives \eqref{eq:s_integrated_iwcc}.
\end{proof}

\begin{theorem}[Recovery of the classical contraction LMI]
For every $t \ge 0$ and $\mathbf{x}\in\mathbb R^n$, we have $L^2_{\mathbf M_t}(\delta_{\mathbf x}) \simeq \mathbb R^n$, $\mathcal V_{\delta_{\mathbf x}}(t)=\{0\}$, and $\mathcal H_{\delta_{\mathbf x}}(t)=\mathbb R^n$. If $U \subseteq \mathbb{R}^n$ is an open set and $\mathcal{D} \coloneqq \{ \delta_\mathbf{y}:\mathbf{y} \in U\}$, then for all $\mathbf{x}\in U$,
\begin{equation}
\operatorname{HorTan}_{\delta_{\mathbf x}}(t,\mathcal D) = \mathbb R^n.
\end{equation}
Consequently, the pointwise IWCC on $\mathcal D$ is equivalent to
\begin{equation*}
B_\lambda(t,\mathbf x)\preceq0 \text{ }\text{ for all }t\ge 0,\, \mathbf x\in U.
\end{equation*}
\end{theorem}
\begin{proof}
The linear isomorphism $L^2_{\mathbf M_t}(\delta_{\mathbf x}) \simeq \mathbb R^n$ is given by the evaluation map $v \mapsto v(\mathbf{x})$. Under this identification, $\mathcal{H}_{\delta_\mathbf{x}}(t) = \mathbb{R}^n$ is clear, because the map $C_c^\infty(\mathbb{R}^n) \ni\psi \mapsto \mathbf{M}_t^{-1}(\mathbf{x}) \nabla \psi(\mathbf{x}) \in \mathbb{R}^n$ is surjective. Hence, $\mathcal{V}_{\delta_{\mathbf x}}(t)=\{0\}$. The claim that $\operatorname{HorTan}_{\delta_{\mathbf x}}(t,\mathcal D) = \mathbb R^n$ follows because every $v \in \mathbb{R}^n$ is in the Wasserstein contingent cone at $\delta_{\mathbf{x}}$. The final claim immediately follows from the definition of the pointwise IWCC.
\end{proof}

\subsection{Analysis of the Motivating Example}
In this subsection, we return to the rotationally symmetric shear example from Section~\ref{subsec:motivating_example}. The goal is twofold.
First, we use this example to argue that the relevant tangent object is precisely $\operatorname{HorTan}_\mu(t,\mathcal C)$, because neither the Wasserstein contingent cone $\operatorname{Tan}_\mu(\mathcal C)$ nor the weighted horizontal space $\mathcal H_\mu(t)$ alone gives the correct test set. This point is illustrated in Claims~\ref{claim:hortan_is_good_def_1} and \ref{claim:hortan_is_good_def_2}.
Second, we prove a symmetry reduction principle showing that, for a metric $\mathbf{M}_t$ preserved under pullbacks by a $G$-action ($G$ being a compact Hausdorff group) and a class $\mathcal{C}\subseteq \mathcal{P}_2$ of $G$-invariant measures, the horizontal tangent cone can be computed using gradients of $G$-invariant scalar potentials. This leads to a family of examples that conceptually generalize the original example of the rotationally symmetric shear: see Claim~\ref{claim:example_compactgroup}.

We take $\mathbf{M}\equiv I$. We use the notations $\mathbf{x} = \left(r \cos(\theta), r \sin(\theta)\right)$, $r \ge 0$ and \begin{equation*}J = \begin{bmatrix}
    0 & -1\\1&0
\end{bmatrix}.\end{equation*}

Let $\sigma \coloneqq \operatorname{Unif}(S^1)$ be the unique rotation-invariant Borel probability measure on the unit circle $S^1$. Clearly, $\sigma \in \mathcal{C}_A$. On $A$, we have $f(\mathbf{x})= - \alpha (r-1) \frac{\mathbf{x}}{r} + (1+ \beta(r-1))J\mathbf{x}$. Therefore, by differentiation, we get \begin{equation*}Df(\mathbf{x})=\begin{bmatrix} -\alpha x_1^2 - \beta x_1 x_2 & -\alpha x_1 x_2 - \beta x_2^2 -1 \\ -\alpha x_1x_2 + \beta x_1^2 +1 & -\alpha x_2^2 +\beta x_1 x_2
\end{bmatrix}
\end{equation*}
for all $\mathbf{x} \in S^1 \subseteq A$. The claim below shows that the correct test space is neither $\operatorname{Tan}_\sigma(\mathcal{C}_A)$ nor $\mathcal{H}_\sigma(t)$ alone, justifying our definition of $\operatorname{HorTan}$~\eqref{eq:def_hortan}.
\begin{claim} \label{claim:hortan_is_good_def_1} Let $\mathbf{M}\equiv I$, and let $\lambda>0$ be given. Consider the vector fields $v_1 (\mathbf{x}) \coloneqq J\mathbf{x} $, $v_2(\mathbf{x}) \coloneqq  x_2 J \mathbf{x}$ in $L^2(\sigma)$. Then, for all $t \ge 0$, we have $v_1 \in \operatorname{Tan}_\sigma(\mathcal{C}_A)\setminus\mathcal{H}_\sigma(t)$ and $v_2 \in \mathcal{H}_\sigma(t) \setminus  \operatorname{Tan}_\sigma(\mathcal{C}_A)$. Moreover, neither $v_1$ nor $v_2$ satisfies the (pointwise) IWCC~\eqref{eq:pointwise_iwcc} at $(t,\sigma)$.
\end{claim}
\begin{proof} Fix any $t \ge 0$.
We first consider $v_1 = (\mathbf{x} \mapsto J \mathbf{x}) \in L^2(\sigma)$.
Since $JR = RJ$ for all $R \in \mathrm{SO}(2)$, for each $h>0$ we have $(\mathrm{Id}+hJ)_\# \sigma = (\mathrm{Id}+hJ)_\# R_\# \sigma = R_\# (\mathrm{Id}+hJ)_\# \sigma$ so $(I+hJ)_\# \sigma$ is rotation-invariant. Moreover, $\left\| (\mathrm{Id}+hJ)\mathbf{x}\right\| = \sqrt{1+h^2} \|\mathbf{x}\|$, so $(\mathrm{Id}+hJ)_\# \sigma$ is concentrated on $\sqrt{1+h^2}S^1$, hence is concentrated on $A$ for all small $h>0$. Therefore, $v_1 \in \operatorname{Tan}_\sigma(\mathcal{C}_A)$. Now we prove $v_1 \in \mathcal{V}_\sigma(t)$; since $v_1 \neq0$ in $L^2(\sigma)$, we conclude that $v_1 \notin \mathcal{H}_\sigma(t)$. For each $\psi \in C_c^\infty(\mathbb{R}^2)$, we have $\int \psi (e^{sJ} \mathbf{x}) d\sigma(\mathbf{x}) = \int \psi (\mathbf{x}) d \sigma (\mathbf{x})$ for all $s \in \mathbb{R}$. Differentiating at $s=0$, we obtain $\int_{\mathbb{R}^2} \nabla \psi \cdot v_1 d \sigma = 0$. Therefore, $v_1 \in \operatorname{Tan}_\sigma(\mathcal{C}_A)\setminus\mathcal{H}_\sigma(t)$. We now consider $v_2$. Consider a $\psi \in C_c^\infty(\mathbb{R}^2)$ such that $\psi(\mathbf{x}) = -x_1/|\mathbf{x}|$ on an open neighborhood of $S^1$. On $S^1$, we have $\nabla \psi (\mathbf{x}) = (x_1^2 -1,x_1x_2)=(-x_2^2,x_1x_2)=v_2(\mathbf{x})$. Therefore, $\nabla \psi = v_2 \in L^2(\sigma)$ and hence $v_2 \in \mathcal{H}_\sigma(t)$. Now we prove $v_2 \notin  \operatorname{Tan}_\sigma(\mathcal{C}_A)$. We have for all $h>0$
\begin{equation*}
\begin{split}
    \int_{\mathbb{R}^2} x_1  d\big((\mathrm{Id}+hv_2)_\# \sigma\big) (\mathbf{x})  &= \int_{\mathbb{R}^2} (x_1 -hx_2^2) d \sigma(\mathbf{x}) \\&= -h \int x_2^2 d \sigma(\mathbf{x})=-\frac h 2,
    \end{split}
\end{equation*}
whereas $\int_{\mathbb{R}^2} x_1 d \mu(\mathbf{x}) =0$ for all $\mu \in \mathcal{C}_A$.
As $\mathbf{x}\mapsto x_1$ is 1-Lipschitz, by the Kantorovich--Rubinstein theorem \cite[Theorem~1.14]{Vi2003}, $W_2\big((\mathrm{Id}+hv_2)_\# \sigma, \mathcal{C}_A \big) \ge W_1\big((\mathrm{Id}+hv_2)_\# \sigma, \mathcal{C}_A \big) \ge h/2$.
Therefore, $v_2 \in \mathcal{H}_\sigma(t) \setminus \operatorname{Tan}_\sigma(\mathcal{C}_A)$.

Now, we test the IWCC~\eqref{eq:pointwise_iwcc} for $v_1$ and $v_2$. Note that $Df(\mathbf{x})J\mathbf{x}=-\mathbf{x}$ on $S^1$, and hence $(J \mathbf{x})^\top B_\lambda(t,\mathbf{x}) (J\mathbf{x}) = 2 \lambda \|\mathbf{x}\|^2=2\lambda$ on $S^1$. Therefore, $\int_{\mathbb{R}^2} v_1^\top B_\lambda v_1 d \sigma  = \int_{\mathbb{R}^2} 2\lambda  d\sigma = 2\lambda >0$ and $\int_{\mathbb{R}^2} v_2^\top B_\lambda v_2 d \sigma  =    \int_{\mathbb{R}^2}2 \lambda x_2^2  d\sigma  = \lambda >0$, so \eqref{eq:pointwise_iwcc} is not satisfied.
\end{proof}

Now, we prove that the IWCC~\eqref{eq:pointwise_iwcc} is satisfied for all $t \ge 0$ and $\mu \in \mathcal{C}_A$. It is interesting to first slightly generalize the example.

\begin{proposition}
\label{prop:symmetry_reduction_hortan}
Let $G$ be a compact Hausdorff group, and let $(g, \mathbf{x}) \mapsto g \cdot \mathbf{x}$ be a group action on $\mathbb{R}^n$. Assume that each $\Theta_g \coloneqq (\mathbf{x} \mapsto g \cdot \mathbf{x})$ is $C^\infty$, and also assume that for each multi-index $\alpha \ge 0$, the map $(g,\mathbf{x})\mapsto D_{\mathbf{x}}^\alpha \Theta_g(\mathbf{x})$ is continuous. Fix $t \ge 0$. Let $\mathcal{C} \subseteq \mathcal{P}_2(\mathbb{R}^n)$ satisfy $(\Theta_g)_\# \mu = \mu$ for all $\mu \in \mathcal{C}$ and $g \in G$.

Assume $D\Theta_g(\mathbf{x})^\top\mathbf{M}_t(g \cdot \mathbf{x})D\Theta_g(\mathbf{x})=\mathbf{M}_t(\mathbf{x})$ for all $g \in G$ and $\mathbf{x} \in \bigcup_{\mu \in \mathcal{C}} \operatorname{supp}(\mu)$, and define
\begin{equation}
    \mathcal H_{\mu,G}(t) \coloneqq \overline{\{\mathbf M_t^{-1}\nabla_{\mathbf{x}}\psi\}}^{L^2_{\mathbf M_t}(\mu)}\subseteq \mathcal{H}_\mu (t),
\end{equation}
where $\psi$ ranges over $C_c^\infty(\mathbb R^n)$ functions such that for all $g \in G$ and $\mathbf{x} \in \mathbb{R}^n$, $\psi(g \cdot \mathbf{x})=\psi(\mathbf{x})$.
Then, we have
\begin{equation}
  \operatorname{HorTan}_{\mu}(t,\mathcal{C}) = \mathcal{H}_{\mu,G}(t) \cap \operatorname{Tan}_\mu (\mathcal{C})
\end{equation}
for all $\mu \in \mathcal{C}$.
\end{proposition}
\begin{proof}
Denote the normalized Haar measure on $G$ by $dg$. We refer to \cite{RoFi2023} for a discussion of Haar measures. It suffices to prove the nontrivial inclusion. Fix $\mu \in \mathcal{C}$, and let $v \in \operatorname{HorTan}_\mu (t, \mathcal{C})$. We prove that $v$ is $G$-invariant, where the $G$-action on $L^2_{\mathbf{M}_t}(\mu)$ is defined by:
\begin{equation}
   (g \cdot v) (\mathbf{x}) \coloneqq D\Theta_g(g^{-1} \cdot \mathbf{x})v(g^{-1} \cdot \mathbf{x}).
\end{equation}
Note that $\|g\cdot v\|^2_{L^2_{\mathbf{M}_t}(\mu)} = \|v\|^2_{L^2_{\mathbf{M}_t}(\mu)} $ because $(\Theta_{g^{-1}})_\# \mu = \mu$ so the map is well-defined, and it is easy to check this is a group action. Moreover, if $\mathbf{M}_t^{-1} \nabla \psi_k$ is an approximation sequence for $v$ in $L^2_{\mathbf{M}_t}(\mu)$, then $g \cdot (\mathbf{M}_t^{-1} \nabla \psi_k ) =\mathbf{M}_t^{-1} \nabla(\psi_k \circ \Theta_{g^{-1}})$ so we have $g \cdot v \in \mathcal{H}_\mu(t)$.

Therefore, to prove $g \cdot v =v$, it suffices to prove $g \cdot v - v \in \mathcal{V}_\mu(t)$. Since $v \in \operatorname{Tan}_\mu(\mathcal{C})$, there exists a strictly decreasing sequence $h_k>0$ converging to $0$ and $\mu_k \in \mathcal{C}$ such that $W_2\big((\mathrm{Id}+h_k v)_\# \mu, \mu_k \big)/h_k \to 0$. We denote $\rho_k \coloneqq (\mathrm{Id}+h_k v)_\# \mu$.

Let $g \in G$ and $\phi \in C_c^\infty$. Define $F = \phi \circ \Theta_g - \phi \in C_c^\infty$. Since $\mu$ and $\mu_k$ are $G$-invariant, $\int_{\mathbb{R}^n} F d \mu = \int_{\mathbb{R}^n} F d\mu_k = 0$. By the Kantorovich--Rubinstein theorem, we have
\begin{equation*}
\begin{split}
   \left|  \int \frac{F}{\operatorname{Lip}(F) +1} d\rho_k  -  \int \frac{F}{\operatorname{Lip}(F) +1} d\mu_k \right | &\le W_1 (\rho_k, \mu_k) \\&\le W_2 (\rho_k, \mu_k),
   \end{split}
\end{equation*}
and therefore $ \left|  \int F d\rho_k\right|/h_k \to 0$ as $k \to \infty$.

By Taylor's formula,
\begin{equation*}
\begin{split}
| F(\mathbf{x}+h_kv(\mathbf{x})) - &F(\mathbf{x}) - h_k \nabla F(\mathbf{x}) \cdot v(\mathbf{x})|\\
&\le \frac{n}{2} \left\|D^2F\right\|_{\sup} h_k^2 |v(\mathbf{x})|^2 .
\end{split}
\end{equation*}
Since $v \in L^2(\mu)$ and $\int_{\mathbb{R}^n} F d\mu = 0$, we conclude that 
\begin{equation*}
    \int \!F d \rho_k = \int \!F(\mathbf{x}+h_kv(\mathbf{x}))d\mu(\mathbf{x}) =  h_k \int \! \nabla F \cdot v d\mu + \mathcal{O}(h_k^2).
\end{equation*}
Combining the two estimates, we conclude that $\int \nabla F \cdot v d \mu =0$. Since $D\Theta_g(\mathbf{x})^\top \nabla \phi(g \cdot \mathbf{x}) \cdot v(\mathbf{x}) = (\nabla \phi \cdot (g \cdot v))(g \cdot \mathbf{x})$,
from $(\Theta_g)_\# \mu = \mu$ we conclude that $\int_{\mathbb{R}^n} \nabla \phi \cdot (g\cdot v )d\mu = \int_{\mathbb{R}^n} \nabla \phi \cdot v d \mu$. This finishes the proof of the claim $g \cdot v = v$.

Now we prove $v \in \mathcal{H}_{\mu,G}(t)$. Choose $\psi_k \in C_c^\infty$ such that $h_k\coloneqq \mathbf{M}_t^{-1} \nabla \psi_k \to v$ in $L^2_{\mathbf{M}_t} (\mu)$. We define the averaged functions
\begin{equation}
    \bar \psi_k (\mathbf{x}) \coloneqq \int_G \psi_k(g\cdot \mathbf{x}) dg.
\end{equation}
We have $\bar \psi_k \in C_c^\infty$: smoothness follows by  differentiating under the integral sign, which works because $G$ is compact and $D^\alpha_{\mathbf{x}}\Theta_g(\mathbf{x})$ is jointly continuous for each $\alpha$. Moreover, $\operatorname{supp}(\bar \psi_k)\subseteq \left\{g^{-1} \cdot \mathbf{x}:g \in G,\,\mathbf{x}\in\operatorname{supp}(\psi_k)\right\}$, hence $\operatorname{supp}(\bar \psi_k)$ is compact. Finally, $\bar \psi_k(g \cdot\mathbf{x}) = \bar \psi_k(\mathbf{x})$ because $dg$ is right-invariant.

Define $\bar h_k \coloneqq \mathbf{M}_t^{-1}\nabla \bar \psi_k \in L^2_{\mathbf{M}_t} (\mu)$. By using $D\Theta_g(\mathbf{x})^{-1}=D\Theta_{g^{-1}}(g \cdot \mathbf{x})$, we get
\begin{equation}
\label{eq:bar_hk_as_integral}
    \bar h_k (\mathbf{x})= \int_G (g^{-1} \cdot h_k)(\mathbf{x}) dg\text{ }\,\forall \mathbf{x} \in \bigcup_{\mu \in \mathcal{C}}\operatorname{supp}(\mu).
\end{equation}
Note that the integrand $(g^{-1} \cdot h_k ) (\mathbf{x})$ is jointly continuous in $(g,\mathbf{x})$. From \eqref{eq:bar_hk_as_integral},
\begin{align*}
    \left\|\bar h_k - v\right\|^2_{L^2_{\mathbf{M}_t}(\mu)} &= \int_{\mathbb{R}^n} \left| \int_G (g^{-1}\cdot h_k)(\mathbf{x})-v(\mathbf{x}) dg\right|^2_{\mathbf{M}_t(\mathbf{x})} d \mu \\
    &\le \int_{\mathbb{R}^n} \! \int_G \left|(g^{-1}\cdot h_k)(\mathbf{x})-v(\mathbf{x})\right|^2_{\mathbf{M}_t(\mathbf{x})} dg d \mu \\&= \int_G \|g^{-1}\cdot h_k - v\|_{L^2_{\mathbf{M}_t} (\mu)}^2 dg,
\end{align*}
where $d\mu = d\mu(\mathbf{x})$, by Cauchy--Schwarz and then Tonelli's theorem. Because we have $g^{-1} \cdot v = v$ in $L^2_{\mathbf{M}_t}(\mu)$ for all $g\in G$ and also the $G$-action on $L^2_{\mathbf{M}_t}(\mu)$ is norm-preserving, we conclude that  $\|\bar h_k - v \|_{L^2_{\mathbf{M}_t}(\mu)} \le \|h_k - v \|_{L^2_{\mathbf{M}_t}(\mu)}$. This proves that $v \in \mathcal{H}_{\mu, G}(t)$, and therefore concludes the proof of the proposition.
\end{proof}

From the above proposition, it is straightforward to check the IWCC for the example of rotationally symmetric shear.
\begin{claim} \label{claim:hortan_is_good_def_2}
Let $\mathbf{M} \equiv I$, and let $\lambda \le \alpha$ be given. For all $t \ge 0$ and $\mu \in \mathcal{C}_A$, the IWCC~\eqref{eq:pointwise_iwcc} holds at $(t,\mu)$. In particular, we have
\begin{equation}
   W_2 (\mathsf{S}_{r,t}\mu, \mathsf{S}_{r,t}\nu) \!\le\!  W_{2,\mathcal{C}_A}\!(\mathsf{S}_{r,t}\mu, \mathsf{S}_{r,t}\nu) \!\le\! e^{- \lambda (t-r)}W_{2,\mathcal{C}_A}\!(\mu, \nu) 
\end{equation}
for all $\mu, \nu \in \mathcal{C}_A$ and $t \ge r \ge 0$.
\end{claim}
\begin{proof}
Fix $\mu \in \mathcal{C}_A$ and $t \ge 0$. Let $v \in \operatorname{HorTan}_\mu (t,\mathcal{C}_A)$. By Proposition~\ref{prop:symmetry_reduction_hortan} with $G=\mathrm{SO}(2)$, there exist $\psi_k \in C_c^\infty (\mathbb{R}^2)$ such that $\nabla \psi_k \to v$ in $L^2(\mu)$ as $k \to \infty$ and $\psi_k(R\cdot)=\psi_k(\mathbf{\cdot})$ for all $R \in \mathrm{SO}(2)$. 

Since we have $\psi_k (e^{sJ}\mathbf{x})=\psi_k(\mathbf{x})$ for $s \in \mathbb{R}$, differentiating at $s=0$ gives $ (J\mathbf{x})^\top  \nabla \psi_k (\mathbf{x})=0$. By passing to a subsequence, we get $(J\mathbf{x})^\top v(\mathbf{x})=0$ $\mu$-a.s., and in particular,
\begin{equation}
    v(\mathbf{x})= a (\mathbf{x}) \hat {\mathbf{r}},\,\,\mu\text{-a.s.}
\end{equation}
Here, $a(\mathbf{x}) \coloneqq  v (\mathbf{x})\cdot\hat {\mathbf{r}} \in \mathbb{R}$ and $\hat {\mathbf{r}} \coloneqq 1_{\mathbf{x} \neq 0} \frac{\mathbf{x}}{|\mathbf{x}|}$.

Because $\hat {\mathbf{r}}^\top D_{\mathbf{x}}f(t,\mathbf{x})\hat{\mathbf{r}} = -\alpha$ for all $\mathbf{x} \in A$, we have
\begin{equation*}
    v(\mathbf{x})^\top B_{\lambda}(t,\mathbf{x}) v(\mathbf{x}) = 2(\lambda-\alpha)a(\mathbf{x})^2 \le 0,\,\,\mu\text{-a.s.},
\end{equation*}
and hence $\int_{\mathbb{R}^2}v(\mathbf{x})^\top B_{\lambda}(t,\mathbf{x}) v(\mathbf{x}) d \mu(\mathbf{x}) \le 0$. This completes the proof of the first part. The second part follows from Theorem~\ref{thm:pointwise_iwcc}; here, finiteness of \eqref{eq:kappaT_def} holds because $\bar A \subseteq\mathbb{R}^2$ is compact.
\end{proof}

Note that we could generalize the above argument by replacing $\mathrm{SO}(2)$ with a compact Lie group $G$. For example, the following holds.
\begin{claim}
\label{claim:example_compactgroup}
Let $G$ be a closed subgroup of $O(n)$ acting canonically on $\mathbb R^n$, and let $\mathfrak g\subseteq\mathfrak{so}(n)=\left\{A \in \mathbb{R}^{n \times n } : A^\top =- A \right\}$ be its Lie algebra. Let $K\subseteq\mathbb R^n$ be a Borel set, and set $\mathcal C_{K} = \left\{\mu\in\mathcal P_2(\mathbb R^n):\mu(K)=1,\,\, R_\#\mu=\mu\,\, \forall R\in G\right\}$.
Set $\mathbf M\equiv I$ and set
\begin{equation}
    f(\mathbf{x})=-\nabla U(\mathbf{x})+\sum_{j=1}^N \omega_j(\mathbf{x})A_j\mathbf{x},
\end{equation}
where $U \in C^2(\mathbb{R}^n)$ and $A_j \in \mathfrak g$, $\omega_j\in C^1(\mathbb{R}^n)$ for each $j$.
Assume that there exists $\alpha\in\mathbb{R}$ such that $\xi^\top \nabla^2U(\mathbf{x})\xi\ge \alpha|\xi|^2$ for all $\mathbf{x} \in K$, $\xi\in (\mathfrak g \mathbf{x})^\perp= \left\{\xi\in\mathbb{R}^n: \xi^\top A \mathbf{x}=0\,\,\forall A\in \mathfrak g\right\}$. Choose any $\lambda \le \alpha$. Then, the IWCC~\eqref{eq:pointwise_iwcc} holds at each $(t,\mu)$ with $t \ge 0$ and $\mu\in \mathcal C_{K}$.
\end{claim}
\begin{proof}
By the closed subgroup theorem, $G$ is a compact Lie subgroup. Also, $\mathfrak{g} = \{A \in \mathfrak{so}(n): e^{tA}\in G \text{ for all }t \in \mathbb{R}\}$. Let $v \in \operatorname{HorTan}_\mu(t,\mathcal{C}_K)$. By Proposition~\ref{prop:symmetry_reduction_hortan}, there exist $G$-invariant potentials $\psi_k \in C_c^\infty(\mathbb{R}^n)$ such that $\nabla \psi_k \to v$ in $L^2(\mu)$. By a similar argument as in Claim~\ref{claim:hortan_is_good_def_2}, we have $v(\mathbf{x})^\top Ax = 0$ for all $\mathbf{x}\in\mathbb{R}^n$ and all $A \in \mathfrak{g}$, where $v$ is redefined up to a $\mu$-null set. Fix $1 \le j \le N$, and write $z_j(\mathbf{x}) \coloneqq \omega_j(\mathbf{x})A_j \mathbf{x}$ for all $\mathbf{x}\in\mathbb{R}^n$. Because $v(\mathbf{x})^\top A_j \mathbf{x} \equiv 0$, we have $v(\mathbf{x})^\top Dz_j(\mathbf{x}) v(\mathbf{x})\equiv 0$, and therefore $v(\mathbf{x})^\top B_\lambda (t,\mathbf{x}) v(\mathbf{x}) \equiv - 2 v(\mathbf{x})^\top \nabla^2 U(\mathbf{x})v(\mathbf{x}) + 2 \lambda |v(\mathbf{x})|^2$, and hence is no greater than $2(\lambda-\alpha)|v(\mathbf{x})|^2 \le0$ $\mu$-a.s.; this completes the proof.
\end{proof}

\section{Incremental ISS and Stabilization of Probability Measures}

We now develop a control-system version of the Intrinsic Wasserstein Contraction Theorem. Consider the control system
\begin{equation}
\dot{\mathbf{x}}=f(t,\mathbf{x},\mathbf{u}),
\end{equation}
where $f:[0,\infty)\times\mathbb{R}^n \times \mathbb{R}^m \to \mathbb{R}^n$. \textbf{In this section, we assume that $f$, $D_\mathbf{x} f$, $D_\mathbf{u} f$ are continuous, and that $\mathbf{M}:[0,\infty)\times\mathbb{R}^n\to\mathcal{S}_{++}^n$ is a $C^1$ map.}

For a continuous control field $\mathbf{k}:[0,\infty)\times\mathbb R^n\to\mathbb R^m$ with $D_{\mathbf{x}} \mathbf k$ continuous, set
\begin{equation}
    f^{\mathbf{k}}(t,\mathbf{x}) \coloneqq f(t,\mathbf{x}, \mathbf{k}(t,\mathbf{x})),
\end{equation}  and we denote its flow map by $(\Phi^{\mathbf{k}}_{r,t})_{r \le t}$. We also define the following (continuous) maps:
\begin{equation}
\begin{aligned}
    A^{\mathbf{k}}(t,\mathbf{x}) &\coloneqq D_{\mathbf{x}} f^{\mathbf{k}}(t,\mathbf{x}),\\
    \dot{\mathbf M}^{\mathbf{k}}(t,\mathbf{x}) &\coloneqq \partial_t\mathbf M(t,\mathbf{x}) + D_{\mathbf{x}}\mathbf M(t,\mathbf{x})[f^{\mathbf{k}}(t,\mathbf{x})],\\
    B_\lambda^{\mathbf{k}}(t,\mathbf{x})&\coloneqq \dot{\mathbf M}^{\mathbf{k}}(t,\mathbf{x}) + A^{\mathbf{k}}(t,\mathbf{x})^\top\mathbf M(t,\mathbf{x}) \\
    &+ \mathbf M(t,\mathbf{x})A^{\mathbf{k}}(t,\mathbf{x}) + 2\lambda\mathbf M(t,\mathbf{x}),
\end{aligned}
\end{equation}
where $\lambda \in \mathbb{R}$ is a constant. Whenever $f^{\mathbf{k}}$ is forward-complete, we write $\mathsf S^{\mathbf{k}}_{r,t}=(\Phi^{\mathbf{k}}_{r,t})_\#$ for the induced pushforward map for each $0 \le r \le t <\infty$.

\subsection{Control Homotopies and an Incremental Wasserstein ISS Theorem}
\label{subsec:wasserstein_iss}
A \emph{control homotopy} joining two control fields $\mathbf k^0$ and $\mathbf k^1$ is a continuous map $\bar{\mathbf u}=\bar{\mathbf u}(s,t,\mathbf{x}):[0,1]\times[0,\infty)\times\mathbb R^n\to\mathbb R^m$ for which $\partial_s\bar{\mathbf u}$ and $D_{\mathbf x}\bar{\mathbf u}$ exist and are continuous, and such that $\bar{\mathbf u}(0,\cdot,\cdot)=\mathbf k^0$ and $\bar{\mathbf u}(1,\cdot,\cdot)=\mathbf k^1$.
We write $\bar{\mathbf u}_s(t,\mathbf{x})\coloneqq \bar{\mathbf u}(s,t,\mathbf{x})$ and
\begin{equation}
\eta_s(t,\mathbf{x})\coloneqq\partial_s\bar{\mathbf u}_s(t,\mathbf{x}).
\end{equation}

The next condition is the control-system counterpart of the IWCC~\eqref{eq:s_integrated_iwcc}.

\begin{definition}[control IWCC, $s$-integrated version]
Let $\mathcal C\subseteq\mathcal P_2$, let $\lambda\in\mathbb R$, and let $\bar{\mathbf u}$ be a control homotopy. We say that $(\mathbf M,\bar{\mathbf u})$ satisfies the $s$-integrated control Intrinsic Wasserstein Contraction Condition (cIWCC) on $\mathcal C$ with rate $\lambda$ if, for every $t\ge0$ and every admissible pair $(\mu_s,v_s)_{s\in[0,1]}$ with $\mu_\cdot\subseteq\mathcal C$ and $v\in\mathcal H_{\mu_\cdot}(t)$,
\begin{equation}
\label{eq:integrated_control_iwcc}
\int_0^1\!\!\int_{\mathbb R^n} v_s(\mathbf{x})^\top B_\lambda^{\bar{\mathbf u}_s}(t,\mathbf{x}) v_s(\mathbf{x}) d\mu_s(\mathbf{x}) ds \le0.  
\end{equation}
Note that $B_\lambda^{\bar{\mathbf u}_s}(t,\mathbf{x})$ is jointly continuous in $(s,t,\mathbf{x})$, hence Borel measurable in $(s,\mathbf{x})$.
\end{definition}

To state the ISS analogue of Theorem~\ref{thm:wasserstein_contraction}, we need an auxiliary definition. Let $(\mu_s)_{s\in[0,1]}$ be a narrowly continuous curve in $\mathcal{P}_2$ and let $\boldsymbol\mu(ds,d\mathbf x)=\mu_s(d\mathbf x)ds$. Recall that in Lemma~\ref{lemma:vertical_horizontal_decomp}, $\mathcal{H}_{\mu \cdot}(t)$ is a closed subspace of $L^2_{\mathbf{M}_t} (\boldsymbol{\mu})$. We denote the orthogonal projection onto $\mathcal{H}_{\mu \cdot}(t)$ by $P_{\mathcal{H}_{\mu \cdot}(t)}$.

\begin{definition}[Homotopy forcing]
For a control homotopy $\bar{\mathbf u}$ and $t\ge0$, define the quantity \begin{equation*}
    \gamma_{\bar{\mathbf u}}(t;\mu_\cdot) \coloneqq \begin{cases} \left\| P_{\mathcal{H}_{\mu \cdot}(t)} \big(D_{\mathbf{u}}f \,\eta\big) \right\|_{L^2_{\mathbf M_t}(\boldsymbol\mu)} &\! \!\!\!\text{if } D_{\mathbf{u}}f \,\eta \in L^2_{\mathbf M_t}(\boldsymbol\mu)\\
    \infty &\! \!\!\!\text{otherwise},
\end{cases} \end{equation*}
where $ (s,\mathbf{x})\mapsto D_{\mathbf u}f(t,\mathbf{x},\bar{\mathbf u}_s(t,\mathbf{x}))\eta_s(t,\mathbf{x})$ is abbreviated by $D_{\mathbf{u}}f \,\eta$.

The homotopy forcing on $\mathcal{C} \subseteq \mathcal{P}_2$ at time $t \ge 0 $ is defined as
\begin{equation}
    \Gamma_{\bar{\mathbf u},\mathcal C}(t)
    \coloneqq \sup_{\mu_\cdot\subseteq\mathcal C} \gamma_{\bar{\mathbf u}}(t;\mu_\cdot) 
\end{equation}
where the supremum is taken over all narrowly continuous curves $(\mu_s)_{s \in [0,1]}$ in $\mathcal{C}$.
\end{definition}

The following Dini comparison lemma is needed for the proof of our main ISS theorem. To the best of our knowledge, the lemma in this form has not appeared in the literature before.

\begin{lemma}[Square-root Dini comparison with forcing]
\label{lemma:sqrt_dini_forcing}
Let $A:[a,b] \to [0,\infty)$ be lower semicontinuous from the left: for all $t \in (a,b]$, $A(t) \le \liminf_{s \to t-} A(s)$. Let $\Gamma:[a,b] \to [0,\infty)$ be Lebesgue measurable, and let $\lambda \in \mathbb{R}$. If $D^+ A \le - 2 \lambda A + 2 \sqrt{A} \Gamma$ on $[a,b)$, we have $\sqrt{A(t)} \le e^{-\lambda(t-a)} \sqrt{A(a)} + \int_a^t e ^{-\lambda(t-\tau)} \Gamma(\tau) d \tau$ on $[a,b]$.
\end{lemma}

\begin{proof}
For each fixed $c >0$, define $W_c(t) \coloneqq \sqrt{A(t)+c}$ and $F_c(t) \coloneqq e^{\lambda(t-a)} W_c(t)$. $W_c$ is lower semicontinuous from the left. Since $x \mapsto -\sqrt{x+c}$ is convex, we have the subgradient inequality  $\sqrt{x+c} - \sqrt{y+c} \le \frac 1 2 \frac{x-y}{\sqrt{y+c}}$ for $x,y >-c$. From this, we conclude $D^+ W_c \le \frac {D^+ A}{2 \sqrt {A+c}} \le \frac{-\lambda A +\sqrt{A}\Gamma}{\sqrt{A+c}}$ on $[a,b)$, and hence $D^+ W_c \le - \lambda W_c + |\lambda| \sqrt{c} + \Gamma$ on $[a,b)$.

Moreover, we have $D^+ F_c(t) = W_c (t) \lambda e^{\lambda(t-a)} + \limsup_{h \to 0+} \left[ e^{\lambda(t+h-a)} \frac{W_c(t+h) - W_c(t)}{h} \right] \le W_c(t) \lambda e^{\lambda(t-a)} + e^{\lambda(t-a)} D^+ W_c(t)$, where the last inequality used $e^{\lambda(t+h-a)}>0$. Hence, we have $D^+ F_c (t) \le e^{\lambda (t - a)} ( \Gamma( t) + |\lambda| \sqrt{c}) $ on $[a,b)$.

Fix $s \in (a,b]$. We now prove the claimed inequality $\sqrt{A(s)} \le e^{-\lambda(s-a)} \sqrt{A(a)} + \int_a^s e^{-\lambda(s-\tau)} \Gamma(\tau)d \tau $. Define $q_c(t) \coloneqq e^{\lambda (t - a)} ( \Gamma( t) + |\lambda| \sqrt{c})$ on $[a,s]$. We may assume $\int_a^s q_c < \infty$ because otherwise \eqref{eq:goal_c_delta} will trivially hold. Let $\delta >0$ be given. By the Vitali--Carath\'{e}odory theorem \cite[Theorem~2.25]{Ru1987}, there exists a lower semicontinuous function $r_{c,\delta} $ such that $0 \le q_c \le r_{c,\delta} \le \infty$ and $\int_a^s r_{c,\delta} - \int_a^s q_c < \delta$. Define $R_{c,\delta} (t) \coloneqq \int_a^t r_{c,\delta}$ and $G_{c,\delta} \coloneqq F_c - R_{c,\delta}$ on $[a,s]$.  $R_{c,\delta}$ is continuous because $\int _a^s r_{c,\delta} < \infty$, and hence $G_{c,\delta}$ is lower semicontinuous from the left.

Then, by the subadditivity of $\limsup$, we have $D^+ G_{c,\delta} (r) \le D^+ F_c(r) + \limsup_{h\to 0+} \left[-\frac{R_{c,\delta} (r+h) - R_{c,\delta}(r)}{h} \right] \le D^+F_c(r) -r_{c,\delta}(r) \le D^+F_c(r) -q_c(r)$ for all $r \in [a,s)$, where the second inequality comes from the lower semicontinuity of $r_{c,\delta}$. Therefore, $D^+ G_{c,\delta} \le 0$, and by Lemma~\ref{lemma:comparison_dini}, $G_{c,\delta}(s) \le G_{c,\delta} (a)$.
Therefore, we have
\begin{equation}
\label{eq:goal_c_delta}
   e^{\lambda (s-a)}\sqrt{A(s)+c} - \delta - \int_a^s q_c < \sqrt{A(a)+c}.
\end{equation}

Passing to the limit $\delta \to 0+$ and then $c \to 0+$ concludes the proof.
\end{proof}

We now state and prove our main incremental ISS bound along a control homotopy.
\begin{theorem}[Incremental Wasserstein ISS]
\label{thm:iss_contraction}
Let $\mathcal C\subseteq\mathcal P_2$ be a nonempty subset and let $(\bar{\mathbf{u}}_s)_{s \in [0,1]}$ be a control homotopy. We assume that for each $s \in [0,1]$, $f^{\bar{\mathbf{u}}_s}$ is forward-complete and $\mathcal{C}$ is forward-invariant under the flow $(\mathsf{S}^{\bar{\mathbf{u}}_s}_{r,t})$ for all $t \ge r \ge 0$.

Let $\Omega_{\mathcal{C}}:=\bigcup_{\mu \in \mathcal{C}} \operatorname{supp}(\mu) \subseteq \mathbb{R}^n$, and let $\lambda \in \mathbb{R}$. Assume that for every $T>0$, $\kappa_T<\infty$ and $\chi_T <\infty$ where
\begin{subequations}
\begin{align}
\kappa_T &\coloneqq \sup_{\substack{s\in[0,1]\\0\le \tau\le T\\\mathbf{x}\in\Omega_{\mathcal C}}}
\lambda_{\max} \left(\mathbf M(\tau,\mathbf{x})^{-1/2} B^{\bar{\mathbf{u}}_s}_\lambda (\tau,\mathbf{x}) \mathbf M(\tau,\mathbf{x})^{-1/2}\right), \\
\begin{split}
\chi_T &\coloneqq \sup_{\substack{s\in[0,1]\\ 0\le\tau\le T\\ \mathbf{x}\in\Omega_{\mathcal C}}}
\Bigl[\eta_s(\tau,\mathbf{x})^\top D_{\mathbf u}f(\tau,\mathbf{x},\bar{\mathbf u}_s(\tau,\mathbf{x}))^\top \mathbf M(\tau,\mathbf{x})\\&\qquad\qquad\qquad D_{\mathbf u}f(\tau,\mathbf{x},\bar{\mathbf u}_s(\tau,\mathbf{x})) \eta_s(\tau,\mathbf{x})\Bigr].
\end{split}
\end{align}
\end{subequations}
Assume $\underline{m} I \preceq \mathbf{M}(t, \mathbf{x}) \preceq \overline {m} I$ for all $(t,\mathbf{x}) \in [0,\infty)\times \Omega_{\mathcal C}$ for some positive constants $\underline{m}, \overline {m}>0$. Let $\Gamma:[0,\infty) \to [0, \infty]$ be a Lebesgue measurable map such that $\Gamma_{\bar{\mathbf{u}},\mathcal{C}}(t) \le \Gamma(t)$ for all $t \ge0$.

If $(\mathbf M,\bar{\mathbf u})$ satisfies the cIWCC on $\mathcal C$ with rate $\lambda$, then, we have the bound
\begin{equation}
\begin{split}
 W_{2,\mathcal C} \left( \mathsf S^{\bar{\mathbf u}_0}_{r,t}\mu_0, \mathsf S^{\bar{\mathbf u}_1}_{r,t}\mu_1 \right)
&\le \sqrt{\frac{\overline m}{\underline m}} e^{-\lambda(t-r)} W_{2,\mathcal C}(\mu_0,\mu_1) \\&+\frac{1}{\sqrt{\underline m}} \int _r^ t e^{-\lambda (t-\tau)} \Gamma(\tau) d\tau
\end{split}
\end{equation}
for every $\mu_0, \mu_1 \in \mathcal{C}$ and $t \ge r \ge 0$.

In particular, suppose that $\mu_0$ is an equilibrium under $\bar{\mathbf u}_0$. If $\Gamma(t)\le\overline{\Gamma}$ for all $t\ge0$, $W_{2,\mathcal{C}} (\mu_0, \mu_1) < \infty$, and $\lambda>0$, we have $\limsup_{t \to \infty} W_{2,\mathcal{C}} (\mu_0, \mathsf S^{\bar{\mathbf u}_1}_{r,t}\mu_1 ) \le \frac{\overline{\Gamma}}{\lambda \sqrt{\underline m}}$ for all $r \ge 0$.
\end{theorem}
\begin{remark}
\label{remark:bounded_omega}
    By continuity, $\kappa_T, \chi_T <\infty$ is automatically satisfied if $\Omega_{\mathcal{C}} \subseteq \mathbb{R}^n$ is bounded.
\end{remark}
\begin{proof}
Since $\underline m I\preceq \mathbf M\preceq \overline m I$, it suffices to prove the weighted estimate
\begin{equation}
\begin{split}
\label{eq:desired_weighted_estimate}
W_{2,\mathcal C,\mathbf M_t} \left( \mathsf S^{\bar{\mathbf u}_0}_{r,t}\mu_0, \mathsf S^{\bar{\mathbf u}_1}_{r,t}\mu_1 \right) &\le e^{-\lambda(t-r)} W_{2,\mathcal C,\mathbf M_r}(\mu_0,\mu_1) \\&+ \int_r^t e^{-\lambda(t-\tau)}\Gamma(\tau)d\tau
\end{split}
\end{equation}
for all $\mu_0,\mu_1\in\mathcal C$ and $0\le r<t$. Fix $0\le r<t<T$. First, note that for every narrowly continuous curve $\nu_\cdot\subseteq\mathcal C$ and every $\tau\in[0,T]$ we have
\begin{equation}
\gamma_{\bar{\mathbf u}}(\tau;\nu_\cdot) \le\left\|D_{\mathbf u}f(\tau,\cdot,\bar{\mathbf u}_\cdot(\tau,\cdot))\eta_\cdot(\tau,\cdot) \right\|_{L^2_{\mathbf M_\tau}(\boldsymbol\nu)}  \le \sqrt{\chi_T}.
\end{equation}
Therefore, by replacing $\Gamma(\tau)$ with $\Gamma(\tau) \wedge \sqrt{\chi _T}$ if necessary, we may assume that $\Gamma<\infty$ on $[0,T]$.

Let a narrowly continuous curve $(\mu_s)_{s\in[0,1]}$ in $\mathcal C$ from $\mu_0$ to $\mu_1$ be given. For every $\tau\in[r,t]$ and $s \in [0,1]$, set $\mu_s^\tau \coloneqq \left(\Phi^{\bar{\mathbf u}_s}_{r,\tau}\right)_\#\mu_s$. By the proof of Lemma~\ref{lemma:pushforward_admissible}, $\mu^\tau_\cdot \subseteq\mathcal{C}$ is narrowly continuous. Define $A:[r,t]\to[0,\infty]$ by
\begin{equation}
A(\tau) \coloneqq \mathcal A_{\mathbf M_\tau}(\mu_\cdot^\tau).
\end{equation}
We aim to apply Lemma~\ref{lemma:sqrt_dini_forcing} to the function $A$, thereby proving
\begin{equation}
\label{eq:sqrtdini_lemma_result}
\sqrt{A(t)} \le e^{-\lambda(t-r)}\sqrt{A(r)} + \int_r^t e^{-\lambda(t-\tau)}\Gamma(\tau) d\tau.
\end{equation}
Note that we may assume $A(r) < \infty$.

Hereafter, we check the assumptions of Lemma~\ref{lemma:sqrt_dini_forcing}. We first prove an auxiliary estimate for the $s$-dependent flow. Fix $s \in [0,1]$, $\tau \in [0,T]$, $\mathbf{x}\in\Omega_{\mathcal C}$, and $\xi\in\mathbb R^n$. For $0 \le h\le T-\tau$, define
\begin{equation}
\begin{aligned}
\tau_h & \coloneqq \tau+h, \\
    \mathbf y_h & \coloneqq \Phi^{\bar{\mathbf u}_s}_{\tau,\tau+h}(\mathbf x), \\
    \xi_h &\coloneqq D_{\mathbf x}\Phi^{\bar{\mathbf u}_s}_{\tau,\tau+h}(\mathbf x)\xi + \partial_s\Phi^{\bar{\mathbf u}_s}_{\tau,\tau+h}(\mathbf x).
\end{aligned}
\end{equation}
Since $\mathcal C$ is forward invariant and the flow maps $\Phi^{\bar{\mathbf u}_s}_{\tau,\tau+h}$ are continuous, we have $\Phi^{\bar{\mathbf u}_s}_{\tau,\tau+h}(\Omega_{\mathcal C}) \subseteq \Omega_{\mathcal C}$ for all $s\in[0,1]$ and $0\le\tau\le\tau+h\le T$. In particular, $\mathbf{y}_h \in \Omega_{\mathcal{C}}$.

By the variational equations~\cite[Theorem~V.3.1]{Ha2002} of the $s$-parametric vector field $f(t,\mathbf{x},\bar{\mathbf u}(s, t, \mathbf{x})) = f^{\bar{\mathbf u}_s}(t,\mathbf{x})$, $\frac{d \xi_h}{dh} = D_\mathbf{x} f^{\bar{\mathbf{u}}_s} (\tau _h,\mathbf{y}_h) \xi_h + D_\mathbf{u} f (\tau_h, \mathbf{y}_h, \bar{\mathbf u}(s,\tau_h,\mathbf{y}_h))\partial_s \bar{\mathbf{u}}(s,\tau_h,\mathbf{y}_h) 
    =A^{\bar{\mathbf{u}}_s} (\tau_h, \mathbf{y}_h) \xi_h +  D_\mathbf{u} f (\tau_h, \mathbf{y}_h, \bar{\mathbf u}(s,\tau_h,\mathbf{y}_h))\eta_s(\tau_h,\mathbf y_h)$,
and therefore
\begin{equation}
\label{eq:xi_M_xi_derivative}
\begin{split}
    \frac d{dh}  \big[\xi_h^\top  \mathbf M(\tau_h,\mathbf y_h) \xi_h \big] &= \xi_h^\top \big(B_\lambda^{\bar{\mathbf u}_s}(\tau_h, \mathbf{y}_h)-2\lambda \mathbf M (\tau_h, \mathbf{y}_h) \big)\xi_h  \\
    + 2 \xi_h^\top \mathbf{M}(\tau_h, \mathbf{y}_h)  &D_\mathbf{u} f (\tau_h, \mathbf{y}_h, \bar{\mathbf u}(s,\tau_h,\mathbf{y}_h))\eta_s(\tau_h,\mathbf y_h)\\
    &\le   \gamma_T \xi_h^\top \mathbf M(\tau_h,\mathbf y_h) \xi_h +\chi_T, 
    \end{split}
\end{equation}
where $\gamma_T \coloneqq \kappa_T - 2 \lambda +1$. Hence, we conclude that
\begin{equation}
\label{eq:final_growth_bound}
    \xi_h^\top  \mathbf M(\tau_h,\mathbf y_h) \xi_h \le e^{\gamma_T h} \xi^\top \mathbf{M}(\tau,\mathbf{x}) \xi + C_T(h) \chi_T,
\end{equation}
where $C_T(h) = (e^{\gamma_T h} -1)/\gamma_T$ if $\gamma_T \neq 0$ and $C_T(h) = h$ otherwise.

Now we verify that $A$ is finite and lower semicontinuous from the left. Let $\tau \in (r,t]$. For each $0<h \le \tau -r$, if $A(\tau-h) < \infty$, then by Lemma~\ref{lemma:vertical_horizontal_decomp} there exists an element $v^{\tau-h}\in\mathcal H_{\mu^{\tau-h}_\cdot}(\tau-h)$ which attains the infimum for $A(\tau-h)$. Applying \eqref{eq:final_growth_bound} with $\tau  \leftarrow \tau-h$, $\tau+h \leftarrow \tau$, and $\xi \leftarrow v_s^{\tau-h}(\mathbf{x})$, we check that the integrability condition of Lemma~\ref{lemma:pushforward_admissible} is satisfied. Therefore, by pushing forward $v^{\tau-h}$ by \eqref{eq:pushforwarded_velocity_field}, we obtain a velocity field which pairs with $\mu^\tau _\cdot$ to form an admissible pair. Using this vector field as a competitor in \eqref{eq:weighted_action}, then changing variables using $\mu^\tau_s = \mathsf{S}^{\bar{\mathbf u}_s}_{\tau-h,\tau} \mu_s^{\tau-h}$, then applying \eqref{eq:final_growth_bound}, and then using the definition of $v^{\tau-h}$, we obtain \begin{equation}
\label{eq:bound_for_lsc}
    A(\tau) \le e^{\gamma_T h} A(\tau-h)+ C_T(h) \chi_T.
\end{equation}
Note that this inequality holds trivially if we had $A(\tau-h) =\infty$. Taking $h=\tau-r$ in \eqref{eq:bound_for_lsc} gives
\begin{equation}
A(\tau)\le e^{\gamma_T(\tau-r)}A(r)+ C_T(\tau-r) \chi_T <\infty.
\end{equation}
Thus $A$ is finite on $[r,t]$. Moreover, since $e^{\gamma_T h}\to1$ and $C_T(h)\to0$ as $h\to 0+$, letting $h \to 0+$ in \eqref{eq:bound_for_lsc} gives $A(\tau) \le \liminf_{h\to 0+}A(\tau-h)$. Hence $A$ is lower semicontinuous from the left.

It remains to prove $D^+ A \le -2 \lambda A  + 2 \sqrt{A} \Gamma$ on $[r,t)$. Fix $\tau\in[r,t)$, and let $v^\tau\in\mathcal H_{\mu_\cdot^\tau}(\tau)$ be the velocity field achieving $A(\tau)$. For $0<h\le t-\tau$, set
\begin{equation*}
Z_{s,h}(\mathbf x)\coloneqq D_{\mathbf x}\Phi^{\bar{\mathbf u}_s}_{\tau,\tau+h}(\mathbf x)v_s^\tau(\mathbf x)+\partial_s\Phi^{\bar{\mathbf u}_s}_{\tau,\tau+h}(\mathbf x).
\end{equation*}
Applying \eqref{eq:final_growth_bound} with $\xi \leftarrow v^\tau_s(\mathbf{x})$, the integrability condition of Lemma~\ref{lemma:pushforward_admissible} is verified. Therefore, the velocity field obtained by pushing $v^\tau$ forward by \eqref{eq:pushforwarded_velocity_field} pairs with $\mu_\cdot^{\tau+h}$ to form an admissible pair. Using this velocity field as a competitor in \eqref{eq:weighted_action}, and changing variables using $\mu_s^{\tau+h}= \mathsf{S}^{\bar{\mathbf u}_s}_{\tau,\tau+h}\mu_s^\tau$, we obtain $A(\tau +h) \le \int_0^1  \!\! \int_{\mathbb{R}^n} Z_{s,h} (\mathbf{x})^\top  \mathbf{M}(\tau+h, \Phi^{\bar{\mathbf u}_s}_{\tau,\tau+h} (\mathbf{x})) Z_{s,h}(\mathbf{x}) d \mu^\tau_s (\mathbf{x})ds$. Hence,
\begin{equation}
\begin{split}
\frac{A(\tau_h)-A(\tau)}{h} \!\le\! \int_0^1 \!\!\int_{\mathbb R^n}\!&\frac 1 h \Bigl[Z_{s,h}(\mathbf x)^\top \mathbf M\bigl(\tau_h,\Phi^{\bar{\mathbf u}_s}_{\tau,\tau_h}(\mathbf x)\bigr) Z_{s,h}(\mathbf x)\\
&-v_s^\tau(\mathbf x)^\top\mathbf M(\tau,\mathbf x)v_s^\tau(\mathbf x)\Bigr] d\mu_s^\tau(\mathbf x)ds,
\end{split}
\end{equation}
where $\tau_h$ denotes $\tau+h$. By \eqref{eq:final_growth_bound}, the integrand is $d\mu_s^\tau(\mathbf x)ds$-a.e. bounded above by
\begin{equation}
e^{|\gamma_T|T} \left[ |\gamma_T| v_s^\tau(\mathbf x)^\top \mathbf M(\tau,\mathbf x)v_s^\tau(\mathbf x)+\chi_T \right],
\end{equation}
which is $d\mu_s^\tau(\mathbf x)ds$-integrable.
Therefore, we may apply the reverse Fatou lemma with \eqref{eq:xi_M_xi_derivative} to conclude that
\begin{equation}
\begin{split}
D^+A(\tau)\!&\le\! \int_0^1\!\!\int_{\mathbb R^n}\!\!v_s^\tau(\mathbf x)^\top
\!\!\left[B_\lambda^{\bar{\mathbf u}_s}(\tau,\mathbf x)-2\lambda\mathbf M\right]v_s^\tau(\mathbf x)d\mu_s^\tau(\mathbf x)ds\\
&+ 2\int_0^1\!\!\int_{\mathbb R^n} v_s^\tau(\mathbf x)^\top \mathbf M D_{\mathbf u}f \eta_s(\tau,\mathbf x) d\mu_s^\tau(\mathbf x)ds,
\end{split}
\end{equation}
where $D_\mathbf{u} f =D_{\mathbf u}f \left(\tau,\mathbf x,\bar{\mathbf u}_s(\tau,\mathbf x) \right)$ and $\mathbf{M}=\mathbf{M}(\tau,\mathbf{x})$. The first integral is no greater than $-2 \lambda A(\tau)$ by the cIWCC. Moreover, since $v^\tau\in\mathcal H_{\mu_\cdot^\tau}(\tau)$, the second integral equals $\langle v^\tau, D_u f \, \eta_s \rangle _{L^2_{\mathbf M_\tau(\boldsymbol\mu^\tau)}}=\langle v^\tau, P_{\mathcal H_{\mu_\cdot^\tau}(\tau)} (D_u f \, \eta_s ) \rangle_{L^2_{\mathbf M_\tau(\boldsymbol\mu^\tau)}}$, hence is no greater than $\sqrt{A(\tau)} \Gamma(\tau)$ by Cauchy--Schwarz. This proves the desired Dini derivative bound.

Hence, by Lemma~\ref{lemma:sqrt_dini_forcing}, we conclude that \eqref{eq:sqrtdini_lemma_result} holds. Therefore, we have $W_{2,\mathcal C,\mathbf M_t} \left( \mathsf S^{\bar{\mathbf u}_0}_{r,t}\mu_0, \mathsf S^{\bar{\mathbf u}_1}_{r,t}\mu_1 \right) \le \sqrt{A(t)} \le e^{-\lambda(t-r)}\sqrt{A(r)} + \int_r^t e^{-\lambda(t-\tau)}\Gamma(\tau) d\tau$.
Taking the infimum over all narrowly continuous curves $\mu_\cdot\subseteq\mathcal C$ from $\mu_0$ to $\mu_1$ proves the desired weighted estimate \eqref{eq:desired_weighted_estimate}, hence proving the theorem.
\end{proof}

\begin{theorem}[cIWCC, pointwise version]
\label{thm:pointwise_control_iwcc}
Let the assumptions of Theorem~\ref{thm:iss_contraction} hold except for the $s$-integrated cIWCC. Suppose that for every $t\ge 0$, for a.e. $s\in[0,1]$, every $\mu\in\mathcal C$, and every $v\in\operatorname{HorTan}_\mu(t,\mathcal C)$,
\begin{equation}
\label{eq:pointwise_control_iwcc}
    \int_{\mathbb R^n} v(\mathbf{x})^\top B_\lambda^{\bar{\mathbf u}_s}(t,\mathbf{x})v(\mathbf{x}) d\mu(\mathbf{x}) \le 0.
\end{equation}
Here, the $ds$-null set may depend on $t$ but not on $\mu$ or $v$.
Then~\eqref{eq:integrated_control_iwcc} holds, and therefore the conclusions of Theorem~\ref{thm:iss_contraction} also hold.
\end{theorem}

\begin{proof}
Fix $t \ge 0$. Let $(\mu_\cdot,v_\cdot)$ be an admissible pair with $\mu_\cdot\subseteq\mathcal C$ and $v\in\mathcal H_{\mu_\cdot}(t)$. By Lemma~\ref{lemma:horizontal_slicing}, $v_s\in\operatorname{HorTan}_{\mu_s}(t,\mathcal C)$ for a.e. $s \in [0,1]$. Applying~\eqref{eq:pointwise_control_iwcc} to $(\mu_s,v_s)$ and integrating in $s$ gives~\eqref{eq:integrated_control_iwcc}.
\end{proof}

\subsection{Brockett's integrator: Stabilization of probability measures on an axisymmetric Wasserstein class}
In this section, we present an example for which the theory from Section~\ref{subsec:wasserstein_iss} is applied. Consider Brockett's integrator~\cite{Br1983} $\dot{\mathbf{x}}=G(\mathbf{x})\mathbf{u}$, where $\mathbf{x}=(x,y,z)$ and
\begin{equation*} 
G(\mathbf{x})= \begin{bmatrix}
    1 & 0\\
    0 & 1\\
    -y & x
\end{bmatrix},
\end{equation*}
or $\dot x=u_1$, $\dot y=u_2$, $\dot z=xu_2-yu_1$. We denote $r=\sqrt{x^2+y^2}$.

Fix $0<r_0<R<r_1$, $Z>0$, and $\alpha,\gamma>0$, and set $K=\{\mathbf x\in\mathbb R^3:r_0\le r\le r_1,\ |z|\le Z\}$.
Let $\mathcal C_{\rm ax}\subseteq \mathcal P_2(\mathbb R^3)$ be the set of all $\mathrm{SO}(2)$-invariant probability measures supported in $K$. Let the target law be $\mu_\star = \operatorname{Unif}(R\mathbb S^1)$, the unique rotation-invariant Borel probability measure on the circle $R\mathbb{S}^1 \coloneqq \{(x,y,z): x^2+y^2 = R^2, z=0\}$. We have $\mu_\star\in\mathcal C_{\rm ax}$.

Fix $\vartheta\in C_c^\infty((0,\infty))$ with $\vartheta=1$ on $[r_0 /2, 2r_1]$. For each $q \in (r_0, r_1)$, define the smooth feedback
\begin{equation}
\mathbf k_q(\mathbf x) \coloneqq \vartheta(r) \pi_{12}\left(-\alpha(r-q)\mathbf e_r-\frac{\gamma z}{r}\mathbf e_\theta\right),
\end{equation}
where $\pi_{12}(x,y,z)\coloneqq (x,y)$.
On $K$, the closed-loop vector field $F_q (\mathbf{x}) \coloneqq G(\mathbf{x})\mathbf{k}_q (\mathbf{x})$ is given by $F_q = -\alpha(r-q)\mathbf e_r-\frac{\gamma z}{r}\mathbf e_\theta-\gamma z\mathbf e_z$, where $\mathbf e_r, \mathbf e_\theta, \mathbf e_z$ are the unit vector fields for the cylindrical coordinate system. Therefore, the closed-loop dynamics on $K$ are given in cylindrical coordinates as
\begin{equation}
\label{eq:brockett_closed_loop_cylindrical}
\dot r=-\alpha(r-q), \quad \dot\theta=-\frac{\gamma z}{r^2}, \quad \dot z=-\gamma z.
\end{equation}
Note that for each $q \in (r_0, r_1)$, $F_q$ is forward complete on $\mathbb{R}^3$ because we have $|F_q(\mathbf{x})| \le C_q(1+|\mathbf{x}|)$ for some $C_q>0$. Note also that $\mu_\star$ is an equilibrium under the control field $\mathbf k_{R}$.
\begin{claim}
For every $\mu, \nu \in \mathcal{C}_{\rm ax}$, we have $W_{2, \mathcal{C}_{\rm ax}} (\mu,\nu) < \infty$.
\end{claim}
\begin{proof}
Let $\nu\in\mathcal C_{\rm ax}$. We explicitly construct an appropriate curve from $\nu$ to $\mu_\star$. For each $s \in [0,1]$, define the map $T_s(r,\theta,z)\coloneqq \bigl((1-s)r+sR,\theta,(1-s)z\bigr)$ in cylindrical coordinates with $T_s \equiv \mathrm{id}$ on the $z$-axis. Let $\nu_s\coloneqq (T_s)_\#\nu$. Since each $T_s$ commutes with every element of $\mathrm{SO}(2)$ and $T_s(K) \subseteq K$, we have $\nu_\cdot \subseteq \mathcal C_{\rm ax}$. Moreover, as $(s,\mathbf{x})\mapsto T_s(\mathbf{x})$ is continuous on $[0,1] \times K$, $s \mapsto \nu_s$ is narrowly continuous. We have $\nu_0=\nu$ and $\nu_1=\mu_\star$.

Define a Borel vector field $w$ on $[0,1]\times\mathbb{R}^3$ by $w_s(r,\theta,z) \coloneqq \frac{R-r}{1-s}\mathbf e_r -\frac{z}{1-s}\mathbf e_z$ with $w_1 \equiv 0$, $w_s(0,\theta, z) \equiv 0$. Let $\varphi\in C_c^\infty((0,1)\times\mathbb R^3)$. For $s \in [0,1]$, define $H(s) \coloneqq \int_{\mathbb R^3}\varphi(s,\mathbf y)d\nu_s(\mathbf y) = \int_K\varphi\left(s,T_s(\mathbf x)\right)d\nu(\mathbf x)$. Since we have $w_s (T_s(\mathbf x)) = \partial_s T_s(\mathbf x)$ on $[0,1) \times K$, the following holds on $[0,1)$:    
\begin{align*}
H'(s)\!&=\!\int_K\left[\partial_s\varphi\bigl(s,T_s(\mathbf x)\bigr) + \nabla\varphi\bigl(s,T_s(\mathbf x)\bigr) \cdot\partial_sT_s(\mathbf x) \right] d\nu(\mathbf x)\\
\!&=\!\int_{\mathbb R^3} \left[\partial_s\varphi(s,\mathbf y)+\nabla\varphi(s,\mathbf y)\cdot w_s(\mathbf y)\right]d\nu_s(\mathbf y).
\end{align*}
The first equality above holds also at $s=1$ and proves, by dominated convergence, that $H \in C^1([0,1])$. Hence, $\int_0^1 H'(s)ds = H(1)-H(0)=0$, so the continuity equation is satisfied. Finally, for $s \in [0,1)$, $\|w_s\|_{L^2(\nu_s)} ^2 = \int _K \left((R-r)^2+(-z)^2 \right)d \nu(\mathbf{x})<\infty$. Therefore, the pair $(\nu_\cdot, w_\cdot)$ is admissible, and $W_{2,\mathcal{C}_{\rm ax}} ^2 (\nu, \mu_\star) \le \int_0^1\!\!\int_{\mathbb R^3}|w_s|^2 d\nu_sds < \infty$. The claim follows from the triangle inequality (see Proposition~\ref{prop:extended_metric}).
\end{proof}
\begin{claim}
\label{claim:exponential_bound_brockett}
For every $t \ge t_0 \ge 0$, $R_0,R_1 \in (r_0,r_1)$, $\nu_0,\nu_1 \in \mathcal C_{\rm ax}$, and $0\neq\lambda \le \min\{\alpha,\gamma\}$,
\begin{equation}
\begin{split}
W_{2,\mathcal C_{\rm ax}}\left(\mathsf S^{\mathbf k_{R_0}}_{t_0,t}\nu_0, \mathsf S^{\mathbf k_{R_1}}_{t_0,t} \nu_1 \right) &\le e^{-\lambda(t-t_0)} W_{2,\mathcal C_{\rm ax}}(\nu_0,\nu_1)\\
&+ \frac{\alpha}{\lambda}\left(1-e^{-\lambda (t - t_0)}\right)|R_1-R_0|.
\end{split}
\end{equation}
In particular, for every $\nu \in \mathcal C_{\rm ax}$,
\begin{equation}
W_{2,\mathcal C_{\rm ax}}\left(\mathsf S^{\mathbf k_R}_{t_0,t}\nu,\mu_\star \right) \le e^{-\lambda(t-t_0)} W_{2,\mathcal C_{\rm ax}}(\nu,\mu_\star)\text{, } t\ge t_0 \ge 0.
\end{equation}
Note that from $W_2 \le W_{2,\mathcal{C}_{\rm ax}}$, we also obtain upper bounds for the $W_2$ distances.
\end{claim}
\begin{proof}
By \eqref{eq:brockett_closed_loop_cylindrical}, $K$ is forward invariant under the closed-loop flow $\Phi^{\mathbf k_q}_{r, t}$ and $\Phi^{\mathbf k_q}_{r, t} \circ R = R \circ \Phi^{\mathbf k_q}_{r, t}$ on $K$ for all $R \in \mathrm{SO}(2)$, $0 \le r \le t$, and $q \in (r_0,r_1)$. Hence, $\mathcal{C}_{\rm ax}$ is forward invariant under each $\mathsf S^{\mathbf k_q}_{r,t}$. Set $R_s\coloneqq (1-s)R_0+sR_1$ and consider the control homotopy $\bar{\mathbf u}_s(t,\mathbf x)=\mathbf k_{R_s}(\mathbf x)$ which is jointly smooth in $(s,t,\mathbf{x})$. Let $\mathbf{M} \equiv I$. We aim to verify the pointwise cIWCC. Fix $t \ge 0$, $\mu\in\mathcal C_{\rm ax}$, and let $s\in[0,1]$, $v\in\operatorname{HorTan}_\mu(t,\mathcal C_{\rm ax})$ be given. Proposition~\ref{prop:symmetry_reduction_hortan} shows that $v$ lies in the $L^2(\mu)$-closure of gradients of $\mathrm{SO}(2)$-invariant $C_c^\infty(\mathbb{R}^3)$ functions, say $\nabla \psi$. Therefore, we have a convergent sequence $\nabla \psi_n \to v$ in $L^2(\mu)$. Since each $\psi_n$ is $\mathrm{SO}(2)$-invariant, $\nabla \psi_n (\mathbf{x}) \cdot \mathbf e_\theta (\mathbf{x}) = 0$ for all $\mathbf{x} \in \mathbb{R}^3 \setminus \{r=0\}$. Therefore, we have $v\cdot\mathbf e_\theta=0$ $\mu$-a.s. Hence, $v =(v\cdot\mathbf e_r) \mathbf e_r + (v\cdot\mathbf e_z) \mathbf e_z$ $\mu$-a.s. Define $v_r \coloneqq v \cdot \mathbf e_r$ and $v_z \coloneqq v \cdot \mathbf e_z$. As $D_\mathbf{x}F_{R_s}(\mathbf{x}) v = -\alpha v_r \mathbf e_r +(\frac{\gamma z}{r^2} v_r - \frac \gamma r v_z)\mathbf e_\theta - \gamma v_z \mathbf e _z$ for $\mu$-a.s. $\mathbf{x} \in K$, we conclude that
\begin{equation*}
\begin{split}
v^\top B_\lambda^{\bar{\mathbf u}_s}v &= 2v^\top D_{\mathbf{x}}F_{R_s} (\mathbf{x})v + 2\lambda \|v\|^2 \\
&= 2(\lambda-\alpha)v_r^2+2(\lambda-\gamma)v_z^2 \le 0\text{, }\mu\text{-a.s.}
\end{split}
\end{equation*}
Integrating with respect to $\mu$ verifies the pointwise cIWCC in Theorem~\ref{thm:pointwise_control_iwcc}.

Fix any $t \ge 0$. On $K$, we have $\eta=\partial_s\bar{\mathbf u} =\alpha(R_1-R_0)\pi_{12}(\mathbf e_r)$ and $(D_{\mathbf u}f)\eta =\alpha(R_1-R_0)\mathbf e_r$. Since orthogonal projection does not increase the norm, $\Gamma_{\bar{\mathbf u}, \mathcal C_{\rm ax}}(t) \le \alpha |R_1-R_0|$. Finally, note that $\Omega_{\mathcal C_{\rm ax}}\subseteq K$, so by Remark~\ref{remark:bounded_omega}, we can apply Theorem~\ref{thm:pointwise_control_iwcc} with $\Gamma(t) \equiv \alpha |R_1 -R_0|$. This proves the claim.
\end{proof}

\begin{remark}
If $\lambda>0$ in Claim~\ref{claim:exponential_bound_brockett}, we achieve exponential convergence to $\mu_\star$ in the $W_2$ distance.
This does not contradict Brockett's necessary condition~\cite{Br1983} because the target is not the point state $0$ but the equilibrium law $\mu_\star$ supported on the circle $\{r=R, z=0\}$.
\end{remark}

\section{Conclusion}
We developed a differential contraction theory for pushforward dynamics on forward-invariant subsets $\mathcal C$ of the $2$-Wasserstein space. The central geometric object is the weighted horizontal tangent cone (whtc), which combines two features: velocity fields that are invisible to the probability distribution are quotiented out by choosing a minimum-energy representative, and the remaining law movements are restricted to those tangent to $\mathcal C$. Requiring the contraction quadratic form to be nonpositive only on the whtc gives the Intrinsic Wasserstein Contraction Condition (IWCC), which is a sufficient condition for contraction in the intrinsic Wasserstein distance. The condition reduces to classical contraction when $\mathcal{C}$ is the class of Dirac measures. The rotationally symmetric shear illustrates that the distributional dynamics can be contractive even when the point-state dynamics is not contractive, due to the existence of nontrivial infinitesimal stabilizers of probability distributions.

An analytical difficulty arises from our admissible curves being merely absolutely continuous and lying in a possibly nonsmooth subset of the Wasserstein space for which classical tangent spaces are unavailable. We overcame this difficulty by proving an almost-everywhere Euler approximation using the superposition principle for the continuity equation and the Banach-valued Lebesgue differentiation theorem, thereby obtaining the pointwise formulation of the IWCC in terms of Wasserstein contingent cones.

For controlled systems, the control IWCC (cIWCC) yields an incremental Wasserstein ISS bound along control homotopies, with a forcing term determined by the horizontal component of the $s$-derivative of the controlled dynamics. A novel square-root Dini comparison lemma converted the energy estimate into the distance estimate. We applied the ISS theorem to Brockett's integrator, certifying that a smooth feedback exponentially stabilizes every law in an axisymmetric class to the uniform law on a circle.

Several directions remain open for future work. It would be useful to identify broad classes of subsets $\mathcal C$ for which the horizontal tangent cones admit a simple characterization. Also, computational methods for constructing contraction metrics $\mathbf M(t,\mathbf{x})$ in the Wasserstein setting, for example through structured ansatzes, would facilitate controller synthesis. Moreover, while this paper focused on flows induced by pushforwards with respect to an underlying ODE flow on $\mathbb{R}^n$, it would be useful to generalize the theory to general flows on $\mathcal{C}\subseteq\mathcal{P}_2$ which do not arise from such pushforwards. Other interesting directions might come from replacing $\mathbb{R}^n$ by a Riemannian manifold or a Polish space. Such extensions would fully establish intrinsic Wasserstein contraction as a systematic tool for the stability analysis and control of measure-valued dynamical systems.

\bibliographystyle{IEEEtran}
\bibliography{references}
\end{document}